\documentclass[a4paper,leqno,12pt]{amsart}
\usepackage{amssymb,amscd}
\usepackage[latin1]{inputenc}
\usepackage{graphicx}
\usepackage{fullpage}
 \divide\oddsidemargin by 2
\usepackage{enumerate}
\makeatletter
\let\@listlla\list

\def\list#1#2{\@listlla{#1}{#2\itemsep=2pt\parsep=0pt\topsep=3pt plus 1pt minus 1 pt}}
\newcommand{\tridiagram}[6]{{\par\par \centering
\@picture(120,120)(0,0) \put(30,95){\makebox(0,0)[r]{$#1$}}
\put(90,30){\makebox(0,0)[tl]{$#3$}}
\put(90,95){\makebox(0,0)[l]{$#2$}}
\put(60,102){\makebox(0,0)[b]{$#4$}}
\put(102,60){\makebox(0,0)[l]{$#6$}}
\put(50,50){\makebox(0,0)[tr]{$#5$}} \thinlines
\put(40,95){\vector(1,0){40}} \put(95,80){\vector(0,-1){40}}
\put(25,80){\vector(1,-1){55}}
\endpicture\par\par}\noindent\ignorespaces}
\def\@map#1#2[#3]{\mbox{$#1 \colon #2 \longrightarrow #3$}}
\def\map#1#2{\@ifnextchar [{\@map{#1}{#2}}{\@map{#1}{#2}[#2]}}
\newcommand{\Aut}[1]{\mbox{\rm Aut}{(#1)}}
\newcommand{\Out}[1]{\mbox{\rm Out}{(#1)}}
\newcommand{\Inn}[1]{\mbox{\rm Inn}{(#1)}}

\newcommand{\mr}{\mathbb{R}}

\newcommand{\mz}{\mathbb{Z}}

\newcommand{\mn}{\mathbb{N}}

\newcommand{\F}[1]{\ensuremath{\mathbb{F}_{#1}}}

\makeatother
\newtheorem{theorem}{Theorem}[section]
\newtheorem{theo}{Theorem}[]

\newtheorem{defi}{Definition}[]
\newtheorem{proposition}[theorem]{Proposition}
\newtheorem{corollary}[theorem]{Corollary}
\newtheorem{lemma}[theorem]{Lemma}
\theoremstyle{definition}
\newtheorem{definition}[theorem]{Definition}

\newtheorem{remark}[theorem]{Remark}

\usepackage{color}

\title{A new class of a-T-menable Free-by-Free groups}

\author{Fran\c{c}ois Gautero}
\date{\today}
\address{Fran\c{c}ois Gautero, Universit\'e C\^ote d'Azur,
CNRS, LJAD (UMR CNRS 7351), Parc Valrose, 06108
Nice Cedex 2, France} \email{Francois.Gautero@univ-cotedazur.fr}


\begin{document}
\begin{abstract}
We prove that a group
$G^\sigma_{n,k} := F_n \rtimes_\sigma \F{k}$, where $\F{n}$ and $\F{k}$
are respectively rank $n$ and rank $k$ free groups, and $\sigma \colon \F{k} \hookrightarrow \Aut{\F{n}}$
is a monomorphism such that $\sigma(\F{k})$ is a subgroup consisting entirely of polynomially growing automorphisms of $\F{n}$,
acts properly isometrically on a finite dimensional CAT(0) cube complex if some finite-index unipotent subgroup of $\sigma(\F{k})$ admits a Bestvina-Feighn-Handel representative satisfying two additional combinatorial properties.
This is achieved by exhibiting a proper space-with-walls structure in the sense of Haglund-Paulin.
In particular any such group $G^\sigma_{n,k}$ is a-T-menable in the sense of Gromov (equivalently satisfies the Haagerup property).
\end{abstract}

\maketitle

\section*{Introduction}
The Haagerup property is an analytical property on groups
introduced in \cite{Haagerup}, about the existence of a {\em proper conditionally negative definite function}:

\begin{defi}

A {\em conditionally negative definite function} on a topological group $G$ is a function $f \colon G \rightarrow \mr$ such that for any positive
integer $n$,
for any $\lambda_1, \cdots, \lambda_n \in \mr$  with $\displaystyle \sum^n_{i=1} \lambda_i = 0$,
for any $g_1,\cdots,g_n$ in $G$ one has $$\sum_{i,j} {\lambda}_i \lambda_j f(g^{-1}_i g_j) \leq 0.$$

A function $f \colon G \rightarrow \mr$ on a topological group $G$ is {\em proper} if $\displaystyle \lim_{g_i \to \infty} f(g_i) = \infty$.

\end{defi}

In a topological group $G$ a sequence of elements $(g_i)_{i \in \mn} \subset G$ tends to infinity with $i$ if and only if for any compact
$K \subset G$ there is $N \geq 0$ such that for any $i \geq N$, $g_i$ does not belong to $K$.
In this paper we will work with {\em discrete groups}, that
is groups with the discrete topology, equipped with a {\em word-length}:
if $G$ is a group with generating set $S$, the {\em word-length of $\gamma \in G$ with respect to $S$},
denoted by $|\gamma|_S$, is the minimum of the lengths of the words in $S \cup S^{-1}$ which define the element $\gamma$.
Saying that a sequence $(g_i)_{i \in \mn} \subset G$ tends to infinity with $i$ then amounts to saying that the word-lengths of the $g_i$'s tend to infinity
with $i$.

 In \cite{Haagerup} Haagerup property was proved to hold for free groups. It has later been
renewed by the work of Gromov, where it appeared under the term of {\em a-T-menability}.
The origin of this terminology is that on the one hand any amenable group satisfies this property and on the other hand it is a weak converse to
Kazhdan's property T: any group which satisfies both properties is a compact group (a finite group in the discrete case).

\begin{defi}[\cite{Haagerup,Cherix,AW}] 
The group $G$ {\em satisfies the Haagerup property}, or {\em is an a-T-menable group}, if and only if there exists a proper conditionally negative definite
function on $G$.
\end{defi}

We refer the reader to \cite{Cherix} for a detailed background and history of this property. Let us notice that very few is known about the preservation of a-T-menability under extensions:
whereas any semi-direct product $\mbox{(a-T-menable)} \rtimes \mathrm{(amenable)}$ is a-T-menable \cite{Jolissaint}, this is not the case for semi-direct products $\mbox{(a-T-menable)} \rtimes \mbox{(a-T-menable)}$. For instance
$\mz^2 \rtimes SL(2,\mz)$, which has the form $\mbox{(amenable)} \rtimes \mbox{(a-T-menable)}$ since $SL(2,\mz)$ admits a free subgroup of finite index
(a-T-menabilty passes from a finite index subgroup of a group
to the group itself), has relative property (T) and so is not a-T-menable (any conditionally negative definite function is bounded on $\mz^2$, see \cite{delaHarpeValette} - in fact for any free subgroup $\F{k}$
of $SL(2,\mz)$, $\mz^2 \rtimes \F{k}$ is not a-T-menable - see \cite{Burger}). In particular it is not known whether any free-by-free group is a-T-menable.

In \cite{Gautero} we exhibit a first non-linear family of groups of the form (free non abelian)-by-(free non abelian), termed {\em Formanek-Procesi groups}, which are a-T-menable.
More precisely, a Formanek-Procesi group is a semi-direct product $\F{n} \rtimes_\sigma \F{n-1}$ where:

\begin{itemize}
  \item $\F{n} = \langle x_1,\cdots,x_n \rangle$ and
$\F{n-1} = \langle t_1,\cdots,t_{n-1} \rangle$
are the free non-abelian groups respectively of ranks $n$ and $n-1$,
  \item letting $\Aut{\F{n}}$ denote the group of automorphisms
of $\F{n}$ then $\sigma \colon \F{n-1} \hookrightarrow \Aut{\F{n}}$ is the monomorphism defined by: $\sigma(t_j)(x_i) = x_i$ for any $i = 1,\cdots,n-1$
and $j = 1,\cdots,n-1$ whereas $\sigma(t_j)(x_n) = x_n x_j$ for $j=1,\cdots,n-1$.
\end{itemize}

The automorphisms in $\sigma(\F{n-1})$ have {\em linear growth}: this is a particular case of {\em polynomial growth}, where the polynomial function has degree one:

\begin{defi} \label{growth}
Let $G$ be a group with finite generating set $S$. An automorphism $\alpha$ of $G$ has {\em polynomial growth} if and only if
there is a polynomial function $P$ such that, for any $\gamma$ in
$G$, for any $m \in \mz$, $|\alpha^m(\gamma)|_S \leq P(m) |\gamma|_S$.
\end{defi}

The nature of the growth of $\alpha \in \Aut{\F{n}}$ only depends on its class $\overline{\alpha}$ in the group of {\em outer automorphisms} of $\F{n}$,
denoted by $\Out{\F{n}}$. Let us recall that $\Out{\F{n}} = \Aut{\F{n}} / \Inn{\F{n}}$, where $\Inn{\F{n}}$ is the group of {\em inner automorphisms $\alpha_w$ ($w \in \F{n}$):}
$\alpha_w(x) = w^{-1} x w$
for any $x \in \F{n}$.

\begin{defi} \label{growth2}
Let $G$ be a group with finite generating set $S$. An outer automorphism of $G$ has {\em polynomial growth} if and only if some (and hence any) automorphism in the class has.
\end{defi}

The purpose of this paper is to prove that the result of \cite{Gautero} is a particular case of a little bit more general phenomenon. Before the statement, let us recall that a semi-direct product $\F{n} \rtimes_\sigma \F{k}$ only depends, up to isomorphism, on the class of $\sigma(\F{k})$ in $\Out{\F{n}}$. Our construction relies upon the profound structure theorem of Bestvina-Feighn-Handel \cite{BFHpolynomial} about subgroups of polynomially growing automorphisms, even if we only appeal to the most elementary results of this theory. The important feature for us is that $\sigma(\F{k})$ admits a finite-index {\em unipotent subgroup}, represented by a pair $(\mathcal F,\Gamma)$ where $\Gamma$ is a {\em filtered graph} and $\mathcal F$ a set of {\em filtered homotopy equivalences} ({\em BFH-representative $(\mathcal F,\Gamma)$}): see subsection \ref{sectionBFH}. We introduce a particular (unfortunately quite restrictive) class of such unipotent subgroups of polynomially growing automorphisms (see Definition \ref{tied}) and prove the following theorem:

\begin{theo} 
\label{the theorem}

Let $k, n$ be two positive integers and let $\F{n}, \F{k}$ respectively denote the rank $n$ and rank $k$ free groups. Let $\Out{\F{n}}$ denote the group of outer automorphisms of $\F{n} = \langle x_1,\cdots,x_n \rangle$.
If $\sigma \colon \F{k} \hookrightarrow \Out{\F{n}}$
is a monomorphism such that $\sigma(\F{k})$ admits a tied, one-sided unipotent subgroup of outer automorphisms of $\F{n}$ then the group
$\F{n} \rtimes_\sigma \F{k}$ acts properly isometrically on some finite dimensional CAT(0) cube complex. In particular it
is a-T-menable in the sense of Gromov.
\end{theo}

Let us immediately emphasize that the above action on the cube complex is not necessarily cocompact. We give a sharp minimal bound for the dimension of the cube complex in Lemma \ref{minimal bound}. A {\em cube complex} is a metric polyhedral
complex in which each cell is isomorphic to the Euclidean cube $[0, 1]^n$
and the gluing maps are isometries. A cube complex is called {\em CAT(0)}
if the metric induced by the Euclidean metric on the cubes turns it
into a CAT(0) metric space (see \cite{Bridson}).
In order to get the above statement, we prove the existence of a
{\em space with walls} structure as introduced by Haglund and Paulin \cite{HaglundPaulin}. A theorem of Chatterji-Niblo \cite{ChatterjiNiblo}
or Nica \cite{Nica} (for similar constructions in other settings, see also \cite{NibloRoller}, \cite{Sageev} or \cite{GH}) gives the announced action on a CAT(0)
cube complex: all this is quickly recalled in Section \ref{space with walls}. A key-point of our structure consists in the introduction, in Section \ref{diagonal walls}, of the so-called ``diagonal walls'', where we need the existence of particular BFH-representative.  



\section{Graphs and free group automorphisms}

\subsection{Graphs and free groups: generalities}
\label{generalites}
A graph $\Gamma$ is a $1$-dimensional CW-complex. A graph is {\em finite} if the number of cells is finite. The $0$-cells are termed {\em vertices}. Each $1$-cell admits two distinct orientations: we denote by $E^+(\Gamma)$ the set of $1$-cells equipped with some chosen orientation, termed {\em positively oriented edges} whereas an {\em (oriented) edge} denotes a $1$-cell equipped with whatever of its two orientations. If $e$ denotes an edge then $e^{-1}$ denotes the edge with the opposite orientation. We denote by $i(e)$ (resp. $t(e)$) the initial (resp. terminal) vertex of an edge $e$. Any set of edges naturally defines a unique set of $1$-cells by forgetting the orientations and any set of $1$-cells is in bijection with the set of p.o. edges so that any terminology for $1$-cells will be used for edges and vice-versa. The {\em valency} of a vertex $v$ in a graph $\Gamma$ is the number of edges which admit $v$ as initial vertex. Our graphs have no valency $2$-vertices. A {\em subgraph $U$ of a graph $\Gamma$} is any $1$-dimensional CW-complex contained in $\Gamma$ (this implies in particular that, if $U$ contains a $1$-cell $e$ then it contains the two $0$-cells to which $e$ is attached). 

A {\em path} (resp. {\em loop}) in a topological space $X$ is the image of a continuous map from the interval (resp. circle) to $X$. A path is {\em reduced} if the map is locally injective. A loop is {\em simple} if the map is an embedding. An {\em edge-path} in a graph is a path between two vertices of the graph which does not backtrack in the interior of the edges,
i.e. the map is locally injective at the points whose images lie in the interior of the edges. It uniquely defines, and is uniquely defined by, the ordered sequence of oriented edges that it crosses. A {\em graph-map} is a continuous map on a graph, which sends vertices to vertices and edges to edge-paths, it is {\em reduced} if the images of the edges are. 

A {\em tree} is a contractible graph, its {\em ends} are its valency $1$-vertices. 
In a tree $\mathcal T$, there is a unique locally injective path between any two points, termed a {\em geodesic}.  This implies in particular that any $1$-cell $E$ cuts the tree in two connected components: the {\em left-side of $E$} (resp. {\em right-side of $E$}) is the connected component containing the initial (resp. terminal) vertex of the associated p.o. edge. More generally:

\begin{definition}
\label{sides}
Let $\mathcal T$ be a tree, let $E$ be a $1$-cell in $\mathcal T$ and let $\mathcal E$ be a set of $1$-cells of $\mathcal T$. The {\em left-side of $E$ w.r.t. $\mathcal E$} (resp. {\em right-side of $E$ w.r.t. $\mathcal E$}) is the union of all the cells in $\mathcal T$ which are connected to the initial vertex of the p.o. edge associated to $E$ by a geodesic crossing an even (resp. odd) number of cells in $\mathcal E$. An edge-path $p$ in $\mathcal T$ {\em spans $E$} (w.r.t. $\mathcal E$) if its initial and terminal vertices lie in distinct sides of $E$ (w.r.t. $\mathcal E$).
\end{definition}

 The universal covering of a finite graph $\Gamma$ is an infinite tree $\mathcal T$ with covering map $\pi_\Gamma \colon \mathcal T \rightarrow \Gamma$. We will always assume that the covering-maps preserve the orientations, that is choosing an orientation on the edges of a lift $\mathcal T$ of a graph $\Gamma$ amounts to choosing an orientation on the edges of $\Gamma$.

We will denote 

\begin{itemize}
  \item by $T^q$ the {\em q-od}, that is the tree with exactly one valency $q$-vertex $v_0$, exactly $q$ ends $v_1,\cdots,v_q$, and $q$ edges $e_i$ in $E^+(T^q)$ oriented from $v_0$ to $v_i$ ($T^q$ is homeomorphic to the cone over $\{v_1,\cdots,v_q\}$),
  \item by $R_q$ the {\em rose with $q$ petals}, that is the graph with exactly one $2q$-valency vertex $v_0$ and $q$ edges in $E^+(R_q)$ with $v_0$ as both initial and terminal vertex,
  \item by $T^\infty_q$ the universal covering of $R_q$, that is the infinite $2q$-valent tree.
\end{itemize}

The fundamental group of a finite graph is a finite rank free group $\F{n}$ so that the universal covering of $\Gamma$ comes equipped with a free, cocompact, left-action of $\F{n}$, the free group being the group of deck-transformations of this covering. More precisely, let us term {\em maximal tree $T$} in $\Gamma$ a tree containing all the vertices of $\Gamma$. We then have the following  

\begin{definition}
\label{ddb}
A {\em  $\F{n}$-tree} $\mathcal T$ is the universal covering  of a finite graph $\Gamma$ together with the choice of a lift $\widetilde{T}_e$  of a maximal tree $T$ in $\Gamma$, and a homotopy-equivalence from $\Gamma$ to $R_n$ which collapses $T$, and carries the positively oriented edges not in $T$, termed {\em essential p.o. edges}, to the positively oriented edges of $R_n$. The {\em essential $1$-cells} of $\mathcal T$ are the lifts of the essential $1$-cells of $\Gamma$ and the {\em maximal trees $\widetilde{T}_w$} in $\mathcal T$ are the lifts of the maximal tree $T$ in $\Gamma$. A $1$-cell in $\mathcal T$ is {\em exceptional} if it is the lift of a $1$-cell which separates $\Gamma$ in two connected components. An {\em EoE-cell of $\mathcal T$} is a $1$-cell which is either essential or exceptional. We denote by $EoE(\Gamma)$ the number of EoE-cells in $\Gamma$.
\end{definition}

We so get an identification of
a subset of the
vertices of $\mathcal T$ with $\F{n}$ by looking at any orbit $\{g . v \}_{g \in \F{n}}$, $v \in V(\widetilde{T}_e)$, we set $g.\widetilde{T}_e = \widetilde{T}_g$. The number of distinct $\F{n}$-orbits of vertices
is equal to the number of vertices in $\Gamma$. Each vertex in $\mathcal T$ so inherits a {\em $\F{n}$-label}, all the vertices in a same lift $\widetilde{T}_{g}$ carrying the same label $g$.

Observe that an edge-path $p = E^{\epsilon_1}_{j_1} E^{\epsilon_2}_{j_2} \cdots E^{\epsilon_r}_{j_r}$ ($\epsilon_i = \pm 1$, $E_{j_i} \in E^+({\mathcal T})$) in the universal covering $\mathcal T$ of a finite graph $\Gamma$ also admits a description by the edges of $\Gamma$ by considering the edge-path $\pi_\Gamma(p) = e^{\epsilon_1}_{j_1} e^{\epsilon_2}_{j_2} \cdots e^{\epsilon_r}_{j_r}$ in $\Gamma$; beware however that $p$ is only defined up to a left $\F{n}$-translation when considering this description by $\pi_\Gamma(p)$. The edges of $\Gamma$ are {\em $\Gamma$-labels} for the edge-path $p$ in $\mathcal T$. An edge-path $p$ is {\em incident to $g \in \F{n}$} if $p$ begins (one says that $p$ is {\em outgoing}) or ends (one says that $p$ is {\em incoming}) at $g$. An edge-path $p$ uniquely defines an element of $\F{n}$ by considering the ordered sequence of essential edges that it crosses since each essential edge is {\em $\F{n}$-labelled} by a generator $\{x^{\pm 1}_1,\cdots,x^{\pm 1}_n\}$ of $\F{n}$. We will sometimes call {\em $\F{n}$-generator} a reduced edge-path in $\mathcal T$ which projects to a simple loop in $\Gamma$ and contains only one essential edge, whose label is then the label of the $\F{n}$-generator. 
The {\em label}, in $\Gamma$ or $\F{n}$, {\em of a $1$-cell in $\mathcal T$} is the union of the labels of the two corresponding edges. 

\begin{definition}

Let $\pi_\Gamma \colon \mathcal T \rightarrow \Gamma$ be some $\F{n}$-tree. 

\begin{enumerate}

  \item An essential $1$-cell is {\em incident} (resp. {\em incoming}, resp. {\em outgoing}) to (resp. at) a subgraph $U$ of $\mathcal T$ if it is incident (resp. incoming, resp. outgoing) to (resp. at) $g \in \F{n}$ which is the label of some $v \in V(U)$, and is not contained in $U$. The set of all the essential cells incident to $U$ is the {\em crown of $U$}, denoted by $Cr(U)$.  The {\em completion} of a subgraph $U$ of $\mathcal T$, denoted by $\overline{U}$, is the smallest tree containing $U$ whose union with the crown of $U$  is also a tree.
  
  \item Two essential cells in the crown of $U$ are {\em coherent}  if one is incoming, the other is outgoing and  they have same $\F{n}$-label. They are {\em twins} if in addition the $\F{n}$-label of the edge-path from the initial vertex of the incoming essential cell to the terminal vertex of the outgoing one is a power of their $\F{n}$-label. 
  
  \item The {\em shift $h_{\gamma,\gamma^\prime}$} at a coherent pair $\{\gamma,\gamma^\prime\}$ is the left-translation which sends $\gamma$ to $\gamma^\prime$.
 \end{enumerate}

\end{definition}

\begin{lemma}
\label{partition}
Let $\mathcal T$ be some $\F{n}$-tree and let $U$ be a subtree of $\mathcal T$. Then each $\gamma \in Cr(U)$ admits exactly one twin, denoted by $\gamma^{tw}$. The collection of all the pairs of essential cells $\{\gamma,\gamma^{tw}\} \subset Cr(U)$ is unique (up to orientation) and defines a partition of $Cr(U)$.
\end{lemma}

\begin{corollary}
\label{corpartition}
With the notations and assumptions of Lemma \ref{partition}, let $\F{n} = \langle x_1,\cdots,x_n \rangle$. For any $\gamma \in Cr(U)$ the shift $h_{\gamma,\gamma^{tw}} \colon \gamma \mapsto \gamma^{tw}$ is uniquely defined and is conjugate to a power of some $x_i$, $1 \leq i \leq n$.
\end{corollary}

\subsection{An invariant tree for unipotent subgroups of $\Out{\F{n}}$} \hfill

\label{sectionBFH}

We recall the notions and results we need from Bestvina-Feighn-Handel theory.

\begin{definition}[\cite{BFHpolynomial}]
An outer automorphism of $\F{n}$ is {\em unipotent} if the automorphism that it induces on $H_1(\F{n};\mz) = \mz^n$ is unipotent.
\end{definition}

\begin{lemma}[\cite{BFHexponentiel},Corollary 5.7.6] 
\label{BFH1}
Any subgroup of polynomially growing outer automorphisms of $\Out{\F{n}}$ admits a unipotent subgroup of finite index.
\end{lemma}

\begin{definition}[\cite{BFHpolynomial}] \hfill
\label{BFHpolynomial}

\begin{enumerate}
  \item A {\em filtered graph of $\F{n}$} is a graph $\Gamma$ with fundamental group isomorphic to $\F{n}$ equipped with a filtration
$\emptyset = \Gamma_0 \subsetneq \Gamma_1 \subsetneq \cdots \subsetneq \Gamma_i \subsetneq \Gamma_{i+1} \subsetneq \cdots \subsetneq \Gamma_r = \Gamma$
where for each $i = 0,\cdots,r-1$, $\Gamma_{i+1}$ is the union of $\Gamma_i$ with a $1$-cell $e_i$.

 \item A homotopy equivalence $f$ of a filtered graph $\Gamma$ is {\em filtered} if it is a graph-map and for any edge $e_i$ of $\Gamma$, $i=1,\cdots r$, $f(e_{i}) = v_{i-1} e_{i} u_{i-1}$ where
$v_{i-1}, u_{i-1}$ are loops contained in $\Gamma_{i-1}$.
\end{enumerate}
\end{definition}

Obviously, the filtration given by Definition \ref{BFHpolynomial} induces an order on the $1$-cells of $\Gamma$, and thus a partial ordering on the $1$-cells in the universal covering. These order and partial ordering will be referred to as {\em the height} of the $1$-cells. The {\em height of a set of $1$-cells} is the maximal height of these cells.
If $\{f_1,\cdots,f_k\}$ is a set of $k$ filtered homotopy equivalences of $\Gamma$, a $1$-cell $e_j$ of $\Gamma$ (or any of its lifts in the universal covering)  is {\em $i$-topmost} if it appears only in $f_i(e_j)$ and in the image under $f_i$ of no other $1$-cell and is {\em topmost} if it is $i$-topmost for each $i \in \{1,\cdots,k\}$. 

\begin{definition}
\label{wb0}
A  filtered graph of $\F{n}$ is {\em well-built} if, for a chosen identification of its fundamental group with $\F{n}$ as given in Definition \ref{ddb},  each essential cell is higher than any $1$-cell in the unique simple loop with same $\F{n}$-label.

\end{definition}

\begin{lemma}
\label{wb}
Any filtered graph admits a maximal tree, and set of essential cells, so that it is well-built.
\end{lemma}

\begin{proof}
Assume the edges are indexed from the lowest to the highest. Starting with $e_1$, consider the maximal sequence of edges $e_1 < e_2 < \cdots < e_{l}$ whose union forms a forest, put them in the maximal tree and define the first following edge $e_{l+1}$ as an essential edge. Then iterate by substituting $e_1$ with $e_{l+2}$.  
\end{proof}

Observe that a filtered homotopy equivalence of a filtered graph $\Gamma$ of $\F{n}$ fixes each vertex of $\Gamma$. We assume fixed
the choice of a $\F{n}$-tree over $\Gamma$.
A filtered homotopy equivalence naturally defines in this way an outer automorphism of $\F{n}$.
The set of all filtered homotopy equivalences
up to homotopy relative to the vertices of $\Gamma$,
equipped with the composition, defines a group (\cite{BFHpolynomial}[Lemma 6.1]) that we denote by $\mathcal F$.
Any choice of $k$ filtered homotopy equivalences of a filtered graph
$\Gamma$ of $\F{n}$ defines a subgroup $\langle f_1,\cdots,f_r \rangle$ of $\mathcal F$ and thus a subgroup $\mathcal U$ of $\Out{\F{n}}$.

\begin{theorem}[\cite{BFHpolynomial}]
\label{BFH}
For any integers $n \geq 2$ and $k \geq 1$, for any rank $k$ unipotent free subgroup $\mathcal U$ of $\Out{\F{n}}$, there exists a a {\em BFH-representative $({\mathcal F},\Gamma)$ of $\mathcal U$} composed of a filtered graph
$\Gamma$ of $\F{n}$
and a family $\mathcal F$ of $k$ reduced filtered homotopy equivalences
of $\Gamma$ defining $\mathcal U$.
\end{theorem}

This concludes the reminders about Bestvina-Feighn-Handel theory. We now give the definition of the particular free subgroups of polynomially growing automorphisms we will be interested in:

\begin{definition}
\label{tied}

Let $(\mathcal F,\Gamma)$ be a well-built BFH-representative. Let $\{F_i\}_{i=1,\cdots,k}$ be reduced lifts of the filtered homotopy equivalences in $\mathcal F$ to the universal covering $\mathcal T$ of $\Gamma$. 

For any $1$-cell $E$ in $\mathcal T$ we denote by $E^i_{top}$ the unique $1$-cell with same label in $F_i(E)$. Let $Bt_i$ (resp. $Bt$) be the map on $\mathcal T$ which to an EoE-cell $E$ (see Definition \ref{ddb}) assigns all the $1$-cells $E^\prime$ such that $F_i(E^\prime)$ contains $E^i_{top}$ (resp.  the set of all $1$-cells $E^\prime$ such that for some $i \in \{1,\cdots,k\}$, $F_i(E^\prime)$ contains $E^i_{top}$).

\begin{enumerate}

\item The BFH-representative $(\mathcal F,\Gamma)$ is {\em one-sided} if, for any EoE-cell $E$, $Bt(E)$ is contained in a single side of $E$. 

\item The BFH-representative $(\mathcal F,\Gamma)$ is {\em tied} if for any EoE-cell $E$ there exists the {\em $E$-tied-set $I_E \subset \{1,\cdots,k\}$} such that 

\begin{itemize}
  \item $E$ is $j$-topmost for any $j \notin I_E$, 
  \item if $R$ denotes a geodesic from $E$ to $\partial \overline{Bt(E)}$ either no $F_i(R)$ contains $E^i_{top}$ or for each $i \in I_E$ $F_i(R)$ contains $E^i_{top}$.
\end{itemize}

\item A unipotent free subgroup of $\Out{\F{n}}$ is {\em tied} (resp. {\em one-sided}) if it admits a tied (resp. one-sided) BFH-representative. 

 
 \end{enumerate}

\end{definition}

\begin{lemma}
\label{fondamental}
With the assumptions and notations of Definition \ref{tied}, assume that $(\mathcal F,\Gamma)$ is one-sided. If $E, X$ are two $1$-cells in $\mathcal T$ such that $F_i(X)$ contains $E^i_{top}$ then $X^i_{top}$ is on the other side of $E^i_{top}$ than $X$ w.r.t. $E$ if and only if the geodesic between $E$ and $X$ crosses an even number of cells in $Bt_i(E)$.
\end{lemma}

\begin{proof}
Let $R$ be the unique geodesic between $E$ and $X$ and assume that it crosses no cell in $Bt_i(E)$. Since $F_i(X)$ contains $E^i_{top}$ and is reduced, $F_i(X)$ contains $F_i(R)$. It follows that $X$ is higher than any cell in $R$ so that $X^i_{top}$ does not belong to $F_i(R)$. It does not belong to $F_i(E)$ for the same reason. Since $R$ crosses no cell in $Bt_i(E)$ and $F_i$ is reduced, $F_i(X)$ ends on the other side of $E^i_{top}$. We so get that $X^i_{top}$ lies on the other side of $E^i_{top}$. The generalization to the case where $R$ crosses an even number of cells in $Bt_i(E)$ is straightforward.
\end{proof}

This gives the following

\begin{corollary}
\label{toutcapourca}
Let $(\mathcal F,\Gamma)$ be a one-sided BFH-representative and let $\{F_i\}_{i=1,\cdots,k}$ be reduced lifts of the filtered homotopy equivalences in $\mathcal F$ to the universal covering $\mathcal T$ of $\Gamma$. For $t \in \F{k}$, let $F_t \in \langle F_1,\cdots,F_k \rangle$ denote the map representing $\sigma(t)$ and for any $1$-cell $E$ let $E^t_{top}$ the unique $1$-cell with same label as $E$ in $F_t(E)$. 

Let $E$ be any EoE-cell in $\mathcal T$ and let $R$ be a geodesic segment from $E$ to $\partial \overline{Bt(E)}$ crossing at least one cell in $Bt(E)$. Then no subgeodesic of $F_t(R)$ in the opposite side of $E^t_{top}$ than $R$ w.r.t. $E$ contains an element of $Bt(E^t_{top})$.
\end{corollary}

Of course, if the properties required in Definition \ref{tied} are satisfied for the EoE-cell $E$ in a fundamental region for the action of $\F{n}$ then they are satisfied for all the left $\F{n}$-translates of $E$. Farther in the paper, we introduce the maps $Bot_i$ and $Bot$ very similar to $Bt_i$ and $Bt$, in another, equivalent, way.

\section{Free by free groups, suspensions}
\label{suspensions}

The aim of this section is to describe the universal covering of the {\em suspension}  (see \ref{def de suspension} below) of a BFH-representative $(\mathcal F,\Gamma)$ and its universal covering termed {\em mapping-cylinder of $(\mathcal F,\Gamma)$}. 

\begin{definition}
Let $\F{n}, \F{k}$ be the rank $n$ and rank $k$ free groups and let $\sigma \colon \F{k} \hookrightarrow \Out{\F{n}}$ be a monomorphism.

The {\em suspension of $\F{n}$ by $\F{k}$ over $\sigma$}, denoted $G^\sigma_{n,k}$, is the group $$G^\sigma_{n,k} := \F{n} \rtimes_\sigma \F{k}.$$

The normal subgroup $\F{n}$ is called the {\em horizontal subgroup} of $G^\sigma_{n,k}$ whereas the subgroup $\F{k}$ is the {\em vertical subgroup}. 
\end{definition}

\begin{definition}
\label{def de suspension}
Let $\Gamma$ be a graph and let $\mathcal F = \{f_1,\cdots,f_k\}$ be a family of continuous maps of $\Gamma$.

The {\em suspension of $\Gamma$ by $\mathcal F$}, denoted by $K_{(\mathcal F,\Gamma)}$, is the $2$-complex $$K_{(\mathcal F,\Gamma)} := (\Gamma \times T^k) / ((x,v_j) \sim (f_j(x),v_0)),$$

The {\em mapping-cylinder of $\Gamma$ under $\mathcal F$} is the universal covering $\pi_K \colon {\mathcal K}_{({\mathcal F},\Gamma)} \rightarrow K_{(\mathcal F,\Gamma)}$ of the suspension of $\Gamma$ by $\mathcal F$.

The {\em $j$-cylinder over a subgraph $U$ of $\Gamma$} is the $2$-complex $$((U \times [0,1]) \sqcup f_j(U)) / \sim \mbox{ with } (x,1) \in (U \times [0,1])  \sim f_j(x) \in f_j(U)$$
\end{definition}

Assume that the maps $f_i$ in Definition \ref{def de suspension} are homotopy equivalences, and thus induce outer automorphisms $\alpha_i$ on the fundamental group of $\Gamma$. If the latter is $\F{n}$, then the fundamental group of the suspension is the semi-direct product $\F{n} \rtimes_\sigma \F{k}$ where the image of the morphism $\sigma \colon \F{k} \hookrightarrow \Out{\F{n}}$ is the subgroup generated by the $\alpha_i$. In particular:

\begin{lemma}
\label{suspension}
Let $\sigma \colon \F{k} \hookrightarrow \Out{\F{n}}$ be a monomorphism such that $\sigma(\F{k})$ is a unipotent subgroup. Let $({\mathcal F},\Gamma)$ be a BFH-representative of $\sigma(\F{k})$. Then the suspension of $\F{n}$ by $\F{k}$ over $\sigma$ is the fundamental group of the suspension of $\Gamma$ by $\mathcal F$: 

$$G^\sigma_{n,k}  = \pi_1(K_{(\mathcal F,\Gamma)}).$$ 
\end{lemma}

In the following lemma, we describe the cells of the suspension complex. This lemma is stated under the assumptions of Lemma \ref{suspension}; what we only really need is the assumption that the maps $f_i$ with which we build the suspension send vertices to vertices. This assumption is satisfied since the filtered homotopy equivalences fix each vertex.

\begin{lemma}
\label{a dire}

With the notations and assumptions of Lemma \ref{suspension}, set ${\mathcal F} = \{f_i\}_{i=1,\cdots,k}$:

\begin{itemize}
  \item The $1$-skeleton of $K_{(\mathcal F,\Gamma)}$ is the union of $\Gamma \times \{v_0\}$ with the union, over the vertices $v$ of $\Gamma$, of the $\{v\} \times R_k$.
 
   \item If $\Gamma$ has $l$ $1$-cells then there are exactly $k.l$ $2$-cells ${(c_{i,j})}^{j=1,\cdots,k}_{i=1,\cdots,l}$ in $K_{(\mathcal F,\Gamma)}$. The boundary $\partial c_{i,j}$ of the p.o. $2$-cell $c_{i,j}$ reads the edge-path $e_i t_j {f_j(e_i)}^{-1} t^{-1}_j$ in the labels of $E^+(R_k) \cup E^+(\Gamma)$. In particular,  $\partial c_{i,j}$ contains exactly two $1$-cells in the union, over the vertices $v$ of $\Gamma$, of the $\{v\} \times R_k$, and one $1$-cell in $\Gamma \times \{v_0\}$ which connects the initial vertices of the associated edges in the copies of $E^+(R_k)$.
 \end{itemize}
 \end{lemma}
 
 Lemma \ref{a dire} allows us to establish the following definitions:
 
 \begin{definition}
 With the notations and assumptions of Lemma \ref{a dire}:
 \begin{itemize}
 
   \item The $1$-cells in $\Gamma \hookrightarrow K_{(\mathcal F,\Gamma)}$ are termed {\em horizontal $1$-cells}. The $1$-cells in the copies of $R_k$ in $K_{(\mathcal F,\Gamma)}$ are termed {\em vertical $1$-cells}. The $2$-cells are called {\em squares}.
   
   Let $c$ be a p.o. square whose boundary reads $e_i t_j {f_j(e_i)}^{-1} t^{-1}_j$ in the labels of $E^+(R_k) \cup E^+(\Gamma)$.

   \item The edge $e_i$ is the {\em bottom} of $c$. The union of the edges in ${f_j(e_i)}^{-1}$ is the {\em top} of $c$.

 \item The edges $t_j$ and $t^{-1}_j$ in $\partial c$, termed {\em $t_j$-cells}, form the {\em vertical boundary} of $c$, $c$ {\em has vertical $j$-boundary} and $c$ is a {\em $j$-square}. 
 
 \end{itemize}
 
 All the above terminology lifts to the mapping-cylinder $\pi_K \colon {\mathcal K}_{({\mathcal F},\Gamma)} \rightarrow K_{(\mathcal F,\Gamma)}$ of $({\mathcal F},\Gamma)$.
\end{definition}

The following lemma is a straightforward consequence of Theorem \ref{BFH} and item (3) of Lemma \ref{orbit-map}:

\begin{lemma}
\label{a savoir}
Let $\sigma \colon \F{k} \hookrightarrow \Out{\F{n}}$ be a monomorphism such that $\sigma(\F{k})$ is a unipotent subgroup. Let $({\mathcal F},\Gamma)$ be a BFH-representative of $\sigma(\F{k})$ and set ${\mathcal F} = \{f_i\}_{i=1,\cdots,k}$. Let $E$ be any horizontal $1$-cell of ${\mathcal K}_{({\mathcal F},\Gamma)}$. 

\begin{enumerate}
  \item Any square $C$ of  ${\mathcal K}_{({\mathcal F},\Gamma)}$ which admits $E$ as bottom contains exactly one $1$-cell in its top which has same $\Gamma$-label as $E$.

 \item For any $i \in \{1,\cdots,k\}$ there is a unique $i$-square $C$ of ${\mathcal K}_{({\mathcal F},\Gamma)}$ whose bottom is a $1$-cell with same $\Gamma$-label as $E$.

 \item The number of $i$-squares which contain $E$ in their top is equal to the number of times the label of $E$ appears in the union of the images of the edges of $\Gamma$ under $f_i \in \mathcal F$. Two distinct such $i$-squares may share a same label for their bottom.

 \end{enumerate}
 
 \end{lemma}

\begin{definition}
\label{a savoir 2}
With the notations and assumptions of Lemma \ref{a savoir}:

\begin{enumerate}

 \item The $1$-cell given by Item $(1)$ of Lemma \ref{a savoir} will be denoted by $E^i_{top}$ if $C$ is a $i$-square. The $1$-cell given by Item $(2)$ of Lemma \ref{a savoir}  will be denoted by $E^i_{bot}$ if $C$ is a $i$-square.
 
 \item  We denote by $Bot_i$ the map which to a given $1$-cell $E$ assigns the union of all the bottoms of the $i$-squares given by Item $(3)$ of Lemma \ref{a savoir}, and by $Bot^*_i$ the map defined by $Bot^*_i(E) = Bot_i(E) \setminus \{E\}$. 
  
  \item The {\em $(i,E)$-lift} is the left-translation ${\mathcal L}_{i,E} \colon E \mapsto E^i_{top}$ by the $G^\sigma_{n,k}$-element which sends $E$ to $E^i_{top}$.

 \end{enumerate}

\end{definition}

 By Lemma \ref{suspension}, the suspension group $G^\sigma_{n,k}$ is the group of deck-transformations of the universal covering $\pi_K \colon {\mathcal K}_{({\mathcal F},\Gamma)} \rightarrow K_{(\mathcal F,\Gamma)}$: 
there is a bijective correspondance between $G^\sigma_{n,k}$ and each orbit of $i$-cells. As for graphs, we fix a lift $\widetilde{T}_e$ of a maximal tree $T$ in $\Gamma$ to get an identification of $G^\sigma_{n,k}$ with each $G^\sigma_{n,k} . v$, $v \in V(\widetilde{T}_e)$. 

\begin{lemma}

\label{orbit-map}

With the notations and assumptions of Lemmas \ref{suspension} and \ref{a dire},  we denote by $F_j$ the lift of $f_j$ to $\mathcal T$ (the universal covering of $\Gamma$) that is $F_j \circ \pi_\Gamma = \pi_\Gamma \circ f_j$ and by $\widetilde{T}_e$ the chosen lift of the maximal tree in $\Gamma$.

Then there is a surjective continuous map $\pi^\infty_k \colon  {\mathcal K}_{({\mathcal F},\Gamma)} \rightarrow T^\infty_k$, called the {\em orbit-map}, of  ${\mathcal K}_{({\mathcal F},\Gamma)}$,  onto the Cayley graph $T^\infty_k$ of the vertical subgroup $F_k$ which satisfies the following properties:

\begin{itemize}
  \item The pre-image of each vertex $w \in \F{k}$ is a copy ${\mathcal T}_w$ of $\mathcal T$ with, by convention, ${\mathcal T}_e$ the lift containing $\widetilde{T}_e$ (which corresponds to $e \in G^\sigma_{n,k}$). Its $1$-cells are horizontal $1$-cells.
  
    \item The pre-image ${\mathcal K}_{w,wt_j}$ of any closed edge from $w$ to $wt_j$ is homeomorphic to the $j$-cylinder over $\mathcal T$. The $1$-cells in the interior of ${\mathcal K}_{w,wt_j}$, which project under $\pi_k$ to the $1$-cells of $T^\infty_k$ are vertical $1$-cells of ${\mathcal K}_{({\mathcal F},\Gamma)}$. The closure of ${\mathcal K}_{({\mathcal F},\Gamma)} \setminus {\mathcal K}_{w,wt_j}$ consists of exactly two connected components, whose union contains all the vertices of ${\mathcal K}_{({\mathcal F},\Gamma)}$.
  
  \item  The boundary of any p.o. square $C_{i,j} \in \pi^{-1}_K(c_{i,j})$ (see Lemma \ref{a dire}) in ${\mathcal K}_{w,wt_j} \subset {\mathcal K}_{({\mathcal F},\Gamma)}$ reads an edge-path of the form $E_i t_j F_j(E^{-1}_i) t^{-1}_j$ in the labels of $E^+(\mathcal T) \cup E^+(R_k)$.

 \end{itemize}
 
 Moreover:
 
 \begin{enumerate}
   \item The orbit of a vertex under the right-action of the vertical subgroup $\F{k}$ of $G^\sigma_{n,k}$ is the lift of $T^\infty_k$ under the orbit-map, passing through this vertex.
   \item For any $w, w^\prime \in \F{k}$, $w.{\mathcal T}_{w^\prime} = {\mathcal T}_{ww^\prime}$ and ${\mathcal T}_{w^\prime} .w = {\mathcal T}_{w^\prime w}$ are free left- and right-actions.
   \item For any $w \in \F{k}$ and cylinder ${\mathcal K}_{w^\prime, w^\prime t_j}$: $w.{\mathcal K}_{w^\prime, w^\prime t_j} = {\mathcal K}_{ww^\prime, ww^\prime t_j}$ is a free left-action.
   \item The stabilizer, for the left and right actions, of both ${\mathcal T}_w$ and ${\mathcal K}_{w^\prime, w^\prime t_j}$ is the horizontal subgroup.
 \end{enumerate}
\end{lemma}

\section{Space with walls structure}
\label{space with walls}

{\em Spaces with walls} were introduced in \cite{HaglundPaulin} in order to check the Haagerup property. When given a set $X$, a {\em wall} of $X$ is a partition $w$ of $X$ in two non-empty classes, also called {\em sides of $w$}. Let $x, y$ be any two points in the set $X$. A wall $w$ of $X$ {\em separates $x$ from $y$} if and only if $x$ and $y$ belong to distinct sides of $w$. When given a family of walls $\mathcal W$, the number of walls in $\mathcal W$ separating two points $x$ and $y$ is denoted by $\mu_{\mathcal W}(x,y)$ and called the {\em wall distance} between $x$ and
$y$ (it might a priori be infinite). If $X$ is contained in an arc-connected topological space $Y$, a path $p$ in $Y$ between two points in $X$ {\em spans a wall $w$} if $w$ separates the endpoints of $p$. 
A {\em space with walls} is a pair $(X,\mathcal W)$ where $X$ is a set and $\mathcal W$ is a family of walls such that for any two distinct points $x, y$ in $X$, $\mu_{\mathcal W}(x,y) < \infty$. 
We say that a discrete group {\em acts properly} on a space with walls $(X,\mathcal W)$ if it leaves invariant $\mathcal W$ and for some (and hence any) $x \in X$ the
function $g \mapsto \mu_{\mathcal W}(x,gx)$ is proper on $G$.

\begin{theorem}[\cite{HaglundPaulin}] 
A discrete group $G$ which acts properly on a space with walls is a-T-menable.
 \end{theorem}
 
 \begin{remark}
 \label{murs classiques}
 
 Let $\F{k}$ be the rank $k$-free group, and let ${T}^\infty_k$ be its Cayley graph for a free basis. If $E$ is any $1$-cell of $T^\infty_k$, if $w_-$ and $w_+$ denote the left and right sides of of $E$ then $(w_-,w_+)$ is a {\em classical $j$-wall for $\F{k}$}. The tree $T^\infty_k$ together with the collection of all the classical walls is a space with walls upon which $\F{k}$ acts properly.
  \end{remark}
 
 In order to get Theorem \ref{the theorem} we will need the (stronger) result below (we refer to \cite{Nica} for a similar statement).
 
 \begin{definition}
 
 Let $(X,\mathcal W)$ be a space with walls. Two walls $(u,u^c) \in \mathcal W$ and $(v,v^c) \in \mathcal W$ {\em cross}
 if all four intersections $u \cap v$, $u \cap v^c$, $u^c \cap v$ and $u^c \cap v^c$
 are non-empty.
 \end{definition}

\begin{theorem}[\cite{ChatterjiNiblo}]
\label{ceci est un rappel}
Let $G$ be a discrete group which acts properly on a space with walls $(X,\mathcal W)$. Then $G$ acts properly isometrically on some CAT(0) cube complex $Cb(\mathcal W)$ whose dimension is equal to the (possibly infinite) supremum of the cardinalities
of finite collections of walls which pairwise cross. In particular $G$ is a-T-menable.\\
\end{theorem}

\fbox{\begin{minipage}{\textwidth}

{\centerline {\bf Notations for Sections from $4$ to $7$ (included):}}  \hfill \\

{\bf $\sigma \colon \F{k} \hookrightarrow \Out{\F{n}}$ is a monomorphism such that $\sigma(\F{k})$ is a tied, one-sided unipotent subgroup (see Definition \ref{tied}). \\

$(\mathcal F, \Gamma)$ is a tied, one-sided BFH-representative of $\sigma(\F{k})$ (see Theorem \ref{BFH} and Definition \ref{tied})) where ${\mathcal F} = \{f_i\}_{i=1,\cdots,k}$, set of reduced filtered homotopy equivalences of $\Gamma$ (see Definition \ref{BFHpolynomial}), lifts to $\{F_i\}_{i=1,\cdots,k}$,  set of reduced graph-maps of $\mathcal T$ with $\pi_\Gamma \colon \mathcal T \rightarrow \Gamma$ the universal covering of $\Gamma$. We assume that $\Gamma$ is well-built (see Definition \ref{wb0} and Lemma \ref{wb}). \\

$G^\sigma_{n,k} = \F{n} \rtimes_\sigma \F{k}$ is the suspension of $\F{n}$ by $\F{k}$ over $\sigma$.\\

${\mathcal K}_{({\mathcal F},\Gamma)}$ is the mapping-cylinder of $\Gamma$ under $\mathcal F$ (see Definition \ref{def de suspension}).\\

$\pi^\infty_k \colon  {\mathcal K}_{({\mathcal F},\Gamma)} \rightarrow T^\infty_k$ is the orbit-map (see Lemma \ref{orbit-map}).}

\end{minipage}}

\section{Vertical walls}

\subsection{Definition and stabilizers} \hfill

By Lemma \ref{orbit-map}, for any $1$-cell $E$ of $T^\infty_k$ $(\pi^\infty_k)^{-1}(E)$ is a cylinder over $\mathcal T$ which cuts ${\mathcal K}_{({\mathcal F},\Gamma)}$ in two connected components, whose union contains all the vertices of ${\mathcal K}_{({\mathcal F},\Gamma)}$. This allows us to give the following

\begin{definition}
\label{mur vertical}
A {\em vertical $j$-wall} ($j \in \{1,\cdots,k\}$) is any pair  $({(\pi^\infty_k)}^{-1}(w_-),{(\pi^\infty_k)}^{-1}(w_+))$ where $(w_-,w_+)$ is a classical $j$-wall for the vertical subgroup $\F{k}$ (see Remark \ref{murs classiques}). The {\em collection of vertical walls} is the collection composed of all the vertical $j$-walls.
\end{definition}

\begin{remark}
\label{elpo}
By definition and Lemma \ref{orbit-map}, for any $j = 1,\cdots,k$ the vertical $j$-walls are in bijection with the $j$-cylinders over $\mathcal T$.
\end{remark}

\begin{lemma}
\label{stabilisateur vertical}
The collection of all the vertical $j$-walls is invariant for the left-action of $G^\sigma_{n,k}$ on ${\mathcal K}_{({\mathcal F},\Gamma)}$. The horizontal subgroup is the left- and right-stabilizer of any vertical $j$-wall.  
\end{lemma}

\begin{proof}
Any element $g$ of $G^\sigma_{n,k}$ uniquely decomposes as $hv$ with $h$ (resp. $v$) in the horizontal (resp. vertical) subgroup. If $\mathcal W =  ({(\pi^\infty_k)}^{-1}(w_-),{(\pi^\infty_k)}^{-1}(w_+))$ is a vertical $j$-wall, $g.\mathcal W =  (hv.{(\pi^\infty_k)}^{-1}(w_-),hv.{(\pi^\infty_k)}^{-1}(w_+))$. By Lemma \ref{orbit-map}, items (3) and (4), $g. \mathcal W = ({(\pi^\infty_k)}^{-1}(v.w_-),{(\pi^\infty_k)}^{-1}(v.w_+))$. Since $(w_-,w_+)$ is a classical $j$-wall for $\F{k}$, $v.(w_-,w_+) = (v.w_-,v.w_+) = (w^\prime_-,w^\prime_+)$ is a classical $j$-wall for $\F{k}$. Thus $g \mathcal W = ({(\pi^\infty_k)}^{-1}(w^\prime_-),{(\pi^\infty_k)}^{-1}(w^\prime_+))$ is a vertical $j$-wall for $G^\sigma_{n,k}$, hence the $G^\sigma_{n,k}$-invariance of the collection of vertical walls. Item (4) of Lemma \ref{orbit-map} gives the horizontal subgroup as left- and -right stabilizer of any vertical $j$-wall.
\end{proof}

\subsection{Finiteness}

\begin{lemma}
\label{pas con}
For any two elements $g=hv$ and $g^\prime = h^\prime v^\prime$ of $G^\sigma_{n,k}$, the orbit-map induces a bijection between the vertical walls separating $g$ from $g^\prime$ and the classical walls of $\F{k}$ between $v$ and $v^\prime$.
\end{lemma}

\begin{proof}
By definition, the classical walls of $\F{k}$ between $v$ and $v^\prime$ are in bijection with the $1$-cells which separate $v$ from $v^\prime$. By Remark \ref{elpo}, the vertical walls are in bijection with the cylinders over $\mathcal T$. For any horizontal elements $h, h^\prime$ and vertical elements $v, v^\prime$ in $G^\sigma_{n,k}$, a cylinder over $\mathcal T$ separates $hv$ from $h^\prime v^\prime$ if and only if it is the pre-image of a $1$-cell separating $v$ from $v^\prime$. Since the horizontal subgroup stabilizes the cylinders over $\mathcal T$, this implies that the vertical walls separating $hv$ from $h^\prime v^\prime$ are in bijection with the classical walls of the vertical subgroup $\F{k}$ separating $v$ from $v^\prime$ and the lemma is proved.
\end{proof}

As a straightforward consequence of this lemma, we get the following

\begin{proposition}
\label{vertical}
Any two elements of $G^\sigma_{n,k}$ are separated by a finite number of vertical walls.
\end{proposition}



  
  


\section{Diagonal walls}
\label{diagonal walls}

Let us recall that an identification of $G^\sigma_{n,k}$ with the orbit of the vertices in some fixed lift, in ${\mathcal K}_{({\mathcal F},\Gamma)}$, of a maximal tree in $\Gamma$ has been chosen. We also recall that the BFH-representative we work with is well-built. This implies in particular that the image of a $\F{n}$-generator contains only once an essential cell of the same label. Since the definition of these diagonal walls involves some tedious definitions, let us begin by giving a sketch of the whole construction:

\subsection{Sketch of the construction of diagonal walls}

Of course, the vertical walls previously defined do not separate the horizontal subgroup $\F{n} \lhd G^\sigma_{n,k}$ so that we now aim to get a proper space with walls structure for the horizontal subgroup.  The discussion below gives an idea of the construction in the case where the graph of the BFH-representative is the rose.

The classical walls of $\F{n}$ are cut by the (exceptional) $1$-cells. Let us consider such a $1$-cell $E$. If it is topmost, then by taking all its images under the subgroup generated by the various lifts one easily gets a system of cuts which defines a wall for $G^\sigma_{n,k}$. Let us consider then the case where $E$ is not topmost. This implies the existence of at least one square $C$ in ${\mathcal K}_{({\mathcal F},\Gamma)}$ which contains $E$ in its top and whose bottom has a different, higher, label. Such a square allows to connect the two vertices of $E$ without passing through $E$. Thus taking all the images of $E$ under the subgroup generated by the lifts, as preceedingly, does not disconnect ${\mathcal K}_{({\mathcal F},\Gamma)}$. One has to ``complete'' this first set of cuts. The strategy to deal with these non-topmost cells (the ``generic'' case) is to add vertical cuts in such a way that the boundary of each square contains an even number of cuts. Once this property satisfied, it then makes sense to define the two sides of the $E$-wall depending on the parity of cuts contained in an edge-path from $g \in G^\sigma_{n,k}$ to the initial element of the associated p.o. edge.

Observe that, if one wishes the left-translates of $E$ under the subgroup generated by the lifts to be cuts, then one does not have the choice of which vertical boundary to be a cut in order to get no more than two cuts in the boundary of our squares. We now refer the reader to Figure \ref{ceststyle}.

\begin{figure}[htbp]
{\centerline{\includegraphics[height=6cm, viewport = 65 420 540 680,clip]{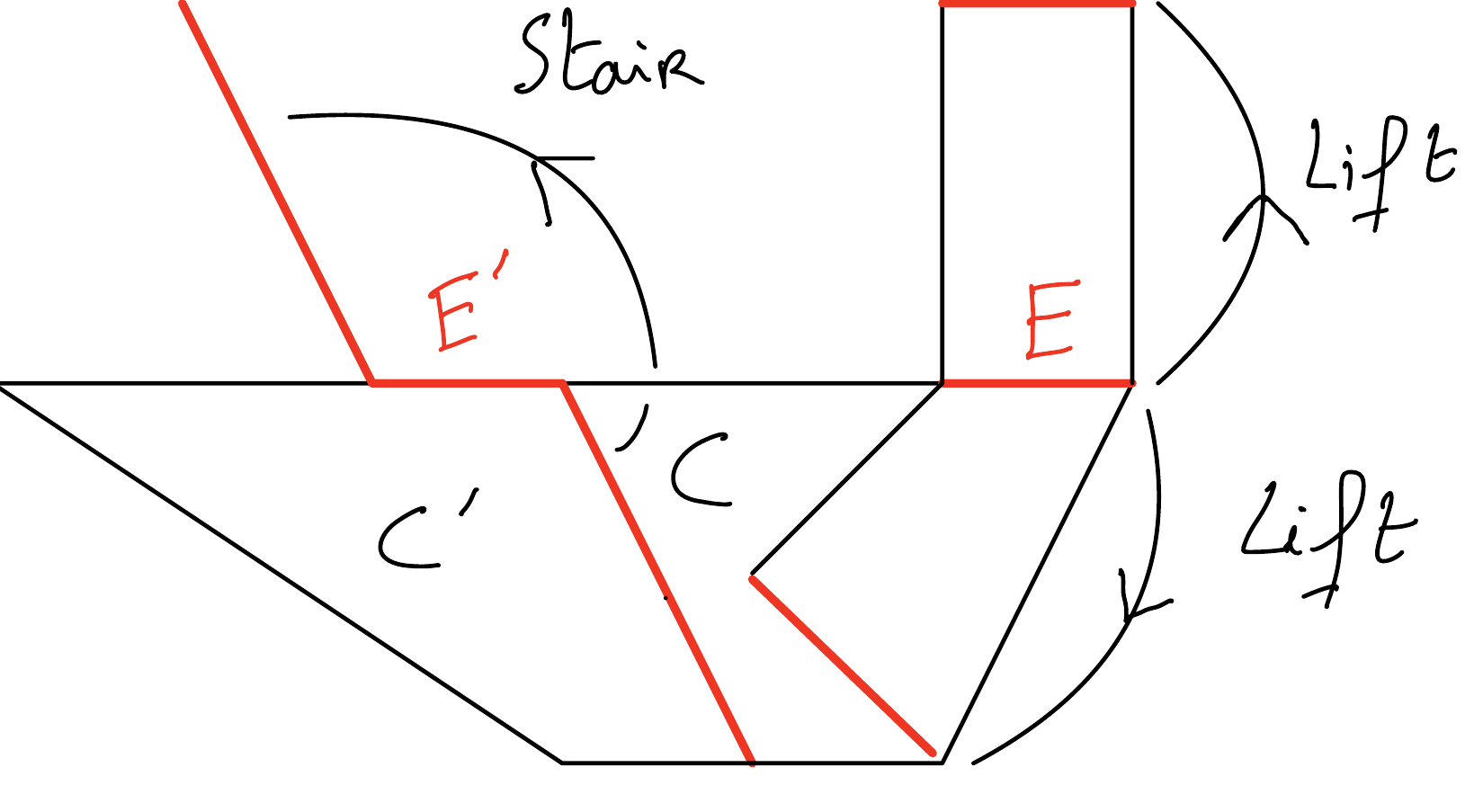}}} \caption{\label{ceststyle} A model for the construction}
\end{figure}

Once defined this vertical boundary in $\partial C$ as a cut, one then has to propagate the cuts by left-translations by horizontal elements, which will at the end lie in the stabilizer of the wall. There is an obstruction: a vertical cut in the left (resp. right) boundary of $C$ is a vertical cut in the right (resp. left) boundary of the adjacent square $C^\prime$ whose bottom has same label. It is thus impossible to translate along the bottom of $C$. The trick is to declare the copy $E^\prime$ of $E$ in the top of $C^\prime$ as a second horizontal cut. However it is impossible to add the lifts of $E^\prime$ without creating squares with odd number of cuts in the boundary: this leads to the definition of the stairs in Definition \ref{atomion}. These elements are defined to carry the vertical cut on the opposite side of $E$ with respect to $E^\prime$ (see Figure \ref{ceststyle}).

In the general case where the graph of the BFH-representative is not the rose, a little bit more care is needed in order to define what it means for an image of a horizontal  generator spanning a given one (like above the image of the bottom of $C$ spans $E$): this is the purpose of the definitions of $Top(E)$, $Bot(E)$ and $B(E)$ and the notion of {\em $E$-spanning twins} in \ref{topshift}. The {\em color} of a vertex is introduced in Definition \ref{couleur}, the cuts being later defined in \ref{cut} as the vertical p.o. edges whose initial vertex have a positive color. The {\em c-bottoms, c-tops and blocks} in Definition \ref{atomion} are the colored version of the tops and bottoms of \ref{topshift}. The {\em biblock} in Definition \ref{oprotein} is the generalization of the passage above from the square $C$ to the union of $C$ with $C^\prime$ above.

The properness of the action comes easily as soon as, at the beginning of the construction, we chose to define as cuts all the $1$-cells in the orbit of our initial horizontal cut under the subgroup generated by the lifts. However, since the bottom of a square containing an essential $1$-cell in its top might be an exceptional $1$-cell, one has to artificially add, to the horizontal cuts, the exceptional $1$-cells. Let us observe that this constraint of putting the whole lift of the initial horizontal cut in the set of cuts is one of the main reasons why we have to restrict to tied one-sided subgroups.

\subsection{Tops, bottoms and colors}

\begin{definition}

\label{topshift}

Let $E$ be an horizontal EoE-cell (see Definition \ref{ddb}) of ${\mathcal K}_{({\mathcal F},\Gamma)}$.

\begin{enumerate}

\item  We denote by $Bot(E)$ the union over $i \in \{1,\cdots,k\}$ of the $Bot_i(E^i_{top})$. The tree $Top_i(E)$ is the top of the $i$-cylinder over the smallest tree containing $Bot(E)$ and we set $$Top(i,E) = \bigcup^{k}_{j=1} ({\mathcal L}_{i,E} \circ {\mathcal L}^{-1}_{j,E}) (Top_j(E))$$. The {\em $E$-top $Top(E)
$} then denotes ${\mathcal L}^{-1}_{i,E}(Top(i,E))$. The {\em $E$-bottom $B(E)$} is the completion of the union over $i \in \{1,\cdots,k\}$ of the sets $V(Top(i,E) t^{-1}_i)$.

\item Let $I$ be a subset of $\{1,\cdots,k\}$. A pair of twins in $Cr(B(E))$ is {\em $E$-spanning} if it spans $Bot^*_i(E^i_{top})$ for some $i \in \{1,\cdots,k\}$, and it is {\em $(I,E)$-spanning} if this holds for any $i \in I$ and no other one.
  
\end{enumerate} 

\end{definition}

\begin{remark}
By definition, the $E$-bottem $B(E)$ constructed above contains the pre-images of all the vertices in all the $Top(i,E)$, $i \in \{1,\cdots,k\}$. This has, as a consequence, that farther in the paper the left-translates of our ``blocks'' (Definition \ref{atomion} farther) are disjoint. This is however not an absolute requisite: what we really need here is only that the left-translates of our ``elementary cuts'' (Definition \ref{cut}) do not belong to the same biblock (Definition \ref{oprotein}).
\end{remark}

\begin{lemma}
\label{analyse}
Let $E$ be an horizontal EoE-cell of ${\mathcal K}_{({\mathcal F},\Gamma)}$. Assume given a pair of twins $\{X,Y\}$ for $B(E)$ and denote by $p_{E,X}, p_{E,Y}$ the unique geodesics from $E$ (so not crossing $E$) respectively to $X$ and $Y$. Then:

\begin{enumerate}
  \item The set $\{1,\cdots,k\}$ decomposes into disjoint subsets $I_X, I_Y, I_{X,Y}, N_{X,Y}$ respectively defined as the sets of indices such that only $p_{E,X}$, resp. only $p_{E,Y}$, resp. both $p_{E,X}$ and $p_{E,Y}$, resp. neither $p_{E,X}$ nor $p_{E,Y}$ span $Bot_i(E^i_{top})$ for each $i \in I_X$, resp. for each $i \in I_Y$, resp. for each  $i \in I_{X,Y}$, resp. for each $i \in N_{X,Y}$. 
  
  \item With the notations of Item $(1)$, the pair $\{X,Y\}$ is $(I_X \cup I_Y,E)$-spanning.
  
  \item The pair $\{X,Y\}$ spans both the cell in $Bot^*_i(E^i_{top})$ which is highest in $p_{E,X}$ for $i \in I_X$ and the cell in $Bot^*_i(E^i_{top})$ which is highest in $p_{E,Y}$ for $i \in I_Y$. In particular, $X$ and $Y$ are. higher than any cell of $Bot^*_i(E^i_{top})$ in $p_{E,X} \cup p_{E,Y}$.

\end{enumerate}
\end{lemma}

\begin{proof}
Item $(1)$ is a tautology and Item $(2)$ a straightforward consequence of the definitions. Let us prove Item $(3)$. Since $E$ is an EoE-cell, $X$ and $Y$ lie in a same side of $E$, without loss of generality we thus assume that $p_{E,X}$ and $p_{E,Y}$ both lie in the right-side of $E$. If $I_X$ is non-empty then for any $i$ in $I_X$, $F_i(p_{E,X})$ ends in the left-side of $E^i_{top}$ whereas $F_i(p_{E,Y})$ ends in the right-side. This implies that the geodesic $q$ from $X$ to $Y$ contains an odd number of cells in $Bot^*_i(E^i_{top})$. By definition of a filtered homotopy equivalence, any one is higher than any other one in the geodesic from $E$ to $q$. The assertion about $X$ and $Y$ comes from the fact that $(\mathcal F,\Gamma)$ is well-built.
\end{proof}

%
%

 \begin{definition}
\label{couleur}
Let $T$ be a horizontal subtree of ${\mathcal K}_{({\mathcal F},\Gamma)}$.

\begin{enumerate}

  \item A {\em $\mathcal E$-color} on $T$  w.r.t. some subset $\mathcal E$ of $1$-cells in $T$ is a map $\chi \colon V(T) \rightarrow \mz/2 \mz$ such that, for any two distinct $v,w$ in $V(T)$ which belong a same $1$-cell $E^\prime$, $\chi(v) + \chi(w) = 1 \mbox{ mod } 2$ if and only if $E^\prime \in \mathcal E$.

 \item A {\em $k$-color} on $T$ is a map $\chi \colon V(T) \rightarrow (\mz/2 \mz)^k$ whose each $i$-coordinate $\chi_i$ is a ${\mathcal E}_i$-color map for some set of $1$-cells ${\mathcal E}_i$ in $T$.
 
 \item A {\em $k$-color on a union of disjoint horizontal trees $T_i$ in  ${\mathcal K}_{({\mathcal F},\Gamma)}$} is a map $\chi \colon V(\bigcup_i T_i) \rightarrow (\mz/2 \mz)^k$ which restricts to a $k$-color on each $T_i$ and satisfies that, if $h$ is a horizontal geodesic between two disjoint $T_i$'s then the colors on the endpoints of $h$ agree.
\end{enumerate}
\end{definition}

\subsection{Blocks}

\begin{definition}

\label{atomion}

Let $E$ be an horizontal EoE-cell of ${\mathcal K}_{({\mathcal F},\Gamma)}$. Let $I= \emptyset$ or $I=I_E$ (the $E$-tied-set)  and let $\epsilon \in \pm$. 

\begin{enumerate}

   \item  The {\em $(\epsilon,I,E)$-c-top}, with $E$ as {\em $(\epsilon,I)$-nucleus}, denotes $Top(E)$ equipped with the $k$-color map defined as follows:
   
   \begin{itemize}
     \item For  any $i \in I$, the $i$-coordinate  is the $E$-color which is positive on the left (resp. right) vertex of $E$ if $\epsilon = +$ (resp. $\epsilon = -$).
     \item For any $l \notin I$, the $l$-coordinate is the $Bot^*_l(E^l_{top})$-color which vanishes on both vertices of $E$.
   \end{itemize}
    
   \item The {\em $(\epsilon,I,E)$-c-bottom}, with $E$ as {\em $(\epsilon, I)$-center}, denotes $B(E)$ equipped with the $k$-color map defined as follows:
   
   \begin{itemize}
   \item If $I = \emptyset$, for any $l \in \{1,\cdots,k\}$ the $l$-coordinate is the $Bot^*_l(E^l_{top})$-color which vanishes on both vertices of $E$.
 
  If $I \neq \emptyset$:  
   
      \item For any $i \in I$, the $i$-coordinate is the $Bot_i(E^i_{top})$-color which is strictly positive on the left (resp. right) vertex of $E$ if $\epsilon = +$ (resp. $\epsilon = -$).
      
   \item For any $l \notin I$,  the $l$-coordinate is the vanishing color.
   \end{itemize}  
   
    \item The {\em $(\epsilon,I,E)$-block ${\mathfrak M}(\epsilon,I,E)$} is the union of the $(\epsilon,I,E)$-c-bottom with:
    
    \begin{itemize}
      \item If $I = I_E$, for each $i \in I$ the $(-\epsilon,I,E^i_{top})$-c-top,
      \item If $I=\emptyset$, for each $l \in \{1,\cdots,k\}$ the $(\epsilon,\emptyset,E^l_{top})$-c-top.
    \end{itemize}
  
The $(\epsilon,\emptyset,E)$-c-top (resp. $(\epsilon,\emptyset,E)$-block) is termed {\em white $E$-c-top} (resp. {\em white $E$-block}).  

A $k$-color map coming with a c-top, c-bottom or block will be termed {\em admissible color}. More generally, a $k$-color map defined on an union of disjoint bottoms $B(.)$ or tops $Top()$ will be termed an {\em admissible color} if its restriction to each such subset makes it into a c-top or c-bottom.

 \item A {\em block-lift of ${\mathfrak M}(\epsilon,I,E)$} is any $(l,E^l_{bot})$-lift ${\mathcal L}_{l,E^l_{bot}}$ for $l \notin I$. 
 
 \item Given a pair of (incoming,outgoing)-twins $(X^-,X^+)$ of the $(\epsilon,I,E)$-c-bottom which spans $E$ and the pair $(X^-_i,X^+_i)$ which spans $E^i_{top}$ for $i \in I$ in the crowns of the corresponding c-tops, we term {\em left $(i,E)$-stair} (resp. {\em right $(i,E)$-stair}) the left-translation by the element ${\mathcal S}_{+,i,E}$ (resp. ${\mathcal S}_{-,i,E}$) which sends $X^+_i$ to $X^-$ (resp. $X^-_i$ to $X^+$).  The term of {\em block-stair} will denote any of these stairs. The integers $l,i$ are {\em the indices of the lifts or stairs considered}.
 
\end{enumerate}

\end{definition}

\begin{lemma}
\label{lemme a la con}
With the assumptions and notations of Definition \ref{atomion}:

\begin{enumerate}

  \item Block-lifts preserve the admissible color.
  
  \item For each $i \in I$, the image of the $(-\epsilon,I,E^i_{top})$-c-top under the block-stair ${\mathcal S}_{\epsilon,i,E}$ is separated from the $(\epsilon,I,E)$-c-bottom by exactly $1$-essential cell.
  
  \item Assume $I = I_E \neq \emptyset$. If one considers the images, for $i \in I_E$ of each $(-\epsilon,I,E^i_{top})$-c-tops under the block-stair ${\mathcal S}_{\epsilon,i,E}$,  the admissible colors of the endpoints of this cell agree so that an admissible color is defined on the union of the block with its images under these block-stairs.

\end{enumerate}

\end{lemma}

\begin{proof}
This is just a careful reading of Definition \ref{atomion}.
\end{proof}

\begin{definition}
\label{companion}
Let $E$ be an horizontal EoE-cell of ${\mathcal K}_{({\mathcal F},\Gamma)}$ and let ${\mathfrak M}$ be the white $E$-block. Let $Z=h_{\gamma,\gamma^{tw}}.E$ where $\{\gamma,\gamma^{tw}\}$ is a $E$-spanning pair and the set of $E$-spanning indices of $\gamma^{tw}$, termed the {\em $Z$-bridge}, is non-empty. The {\em $Z$-companion of ${\mathfrak M}$} is the block ${\mathfrak M}(\epsilon,I_E,,Z^i_{top})$, where $\epsilon$ is chosen so that for $i \in I_E$ the $i$-coordinate of the admissible color is positive on the vertex of $Z$ in the same side of $Z$ as the $Z$-bridge.
\end{definition}

\subsection{Biblocks} \hfill

%
%
%
%
%
%

\begin{definition}
\label{oprotein}
Let $E$ be an horizontal EoE-cell of ${\mathcal K}_{({\mathcal F},\Gamma)}$ and let ${\mathfrak M}(E)$ denote the white $E$-block.  The {\em $E$-biblock ${\mathfrak P}(E)$} denotes:

\begin{enumerate}

  \item If some $Bot^*_i(E^i_{top})$ contains at least one non-exceptional cell, the completion of the union of ${\mathfrak M}(E)$ with a $Z$-companion whose bridge has greatest height.
  
  \item Otherwise, the block ${\mathfrak M}(E)$.
  
\end{enumerate}

The bridge between the two blocks of the $E$-biblock is termed {\em the main bridge}. The intersection of the $E$-biblock with a horizontal tree is a {\em $E$-cluster}.
\end{definition}

\begin{lemma}
\label{coloriage}
Let $E$ be an horizontal EoE-cell of ${\mathcal K}_{({\mathcal F},\Gamma)}$. The admissible colors of the two blocks composing the $E$-biblock ${\mathfrak P}(E)$ define an admissible color of ${\mathfrak P}(E)$.
\end{lemma}

\begin{proof}
This is just a matter of checking that the admissible colors of $\mathfrak M$ and of its companions agree at the endpoints of the main bridge and at the endpoint of any other horizontal geodesic between the two. By definition of the white block and of the admissible color, since $({\mathcal F,\Gamma})$ is tied $I_E$ is the set of positive coordinates of the admissible color of $\mathfrak M$ on the main bridge whereas the other coordinates vanish. Hence the admissible colors agree at the endpoints of the main bridge. Since $(\mathcal F,\Gamma)$ is one-sided Corollary \ref{toutcapourca} gives us that the admissible colors of the endpoints of any other horizontal geodesic connecting one block to the other agree.
\end{proof}

Lemma \ref{coloriage} gives sense to the following definition.

\begin{definition}
\label{appariement}
Let $E$ be an horizontal EoE-cell of ${\mathcal K}_{({\mathcal F},\Gamma)}$.  Let $(X^+,X^-)$ be an (incoming,outgoing)-coherent pair in the crown of a c-top or c-bottom in ${\mathfrak P}(E)$. We say that $(X^+,X^-)$ or the shift $h_{X^+,X^-}$, is {\em enabled} if the admissible color on $i(X^-)$ is equal to the admissible color on $t(X^+)$.  In this case, this admissible color is {\em the admissible color of $(X^+,X^-)$} or {\em of $h_{X^+,X^-}$}. Otherwise  $(X^+,X^-)$, or $h_{X^+,X^-}$ is {\em disabled}.
\end{definition}

\begin{lemma}
\label{okay}
With the assumptions and notations of Definition \ref{appariement}:

\begin{enumerate}

 \item At the exception of the pairs of twins which span the nuclei of the companion block in ${\mathfrak P}(E)$, or a center of a c-bottom, or which contain a bridge distinct from the main bridge of the $E$-biblock, any pair of twins is enabled.
 \item There is a label-preserving bijection between the set of all (incoming,outgoing)-pairs of twins $(X^+,X^-)$ containing a bridge in the companion in ${\mathfrak P}(E)$ and the set of all (incoming,outgoing)-pairs of twins $(Y^+,Y^-)$ containing a bridge in the white block in ${\mathfrak P}(E)$ which satisfies the following property:
 
 \begin{itemize}
 
   \item The admissible color at $t(X^+)$ (resp. $i(X^-)$) agrees with the admissible color at $i(Y^-)$ (resp. $t(Y^+)$).
   
   \item The admissible color of exactly one coherent pair among $(t(X^+),i(Y^-))$ and $(i(X^-),t(Y^+))$ does not vanish.

\end{itemize} 

\end{enumerate}
   
\end{lemma}

\begin{proof}
This is a consequence of the admissible coloring and of the fact that $(\mathcal F,\Gamma)$ is tied and one-sided.
\end{proof}

\begin{definition}
\label{mirrors}
With the assumptions and notations of Item $(2)$ in Lemma \ref{okay}: we term {\em mirror} the shift associated to a coherent pair $(t(X^+),i(Y^-))$ or $(i(X^-),t(Y^+))$.
\end{definition}

\subsection{Diagonal Subgroup and Cuts}

\begin{definition}
\label{diagonalsubgroup}
Let $E$ be an horizontal EoE-cell of ${\mathcal K}_{({\mathcal F},\Gamma)}$.

\begin{enumerate}
  \item The {\em $E$-horizontal subgroup ${\mathfrak H}(E)$} is the subgroup generated by all the enabled shifts of the $E$-biblock with non-vanishing admissible color: they are termed {\em $E$-horizontal generators}. If ${\mathcal C}$ is some $E$-cluster of ${\mathfrak P}(E)$, we denote by ${\mathfrak H}_{\mathcal C}(E)$ the subgroup of ${\mathfrak H}(E)$ generated by the shifts of ${\mathcal C}$ in the $E$-horizontal generators.
  \item The {\em $E$-diagonal subgroup ${\mathfrak D}(E)$} is the subgroup generated by all the $E$-horizontal generators and block-stairs: they are termed {\em $E$-diagonal generators}.
  \item The {\em $E$-replicative subgroup ${\mathfrak R}(E)$} is the subgroup generated by all the block-lifts and all the $E$-diagonal generators.
\end{enumerate}
\end{definition}

The  following, straightforward, lemma characterizes the subgroup of ${\mathfrak R}(E)$ which stabilizes the horizontal tree ${\mathcal T}_{\pi^\infty_k(E)}$ in ${\mathcal K}_{({\mathcal F},\Gamma)}$:

\begin{lemma}
\label{horEstab}
With the assumptions and notations of Definition \ref{diagonalsubgroup}: the stabilizer of ${\mathcal T}_{\pi^\infty_k(E)}$ in ${\mathfrak R}(E)$, termed {\em horizontal $E$-stabilizer}, is the subgroup ${\mathfrak S}_{hor}(E)$ of the horizontal subgroup $\F{n}$ generated by the elements $L^{-l}_1 h L^l_2$, for any horizontal $E$-generator $h$ of ${\mathfrak H}(E)$, and for any pair of block-stairs or lifts $L_1, L_2$ with same index.
\end{lemma}

\begin{definition}
\label{cut}
Let $E$ be an horizontal EoE-cell of ${\mathcal K}_{({\mathcal F},\Gamma)}$.

The {\em elementary $E$-cuts} consist of:

\begin{enumerate}
  \item all the nuclei of the $E$-biblock,
  
  \item all the p.o. $t_i$-cells starting at a vertex whose admissible color is positive.
  \end{enumerate}
  
  The {\em (horizontal, vertical) $E$-cuts} are the left ${\mathfrak R}(E)$-translates of the (horizontal, vertical) elementary $E$-cuts.
\end{definition}

\begin{proposition}
\label{0 ou 2}
Let $E$ be an horizontal EoE-cell of ${\mathcal K}_{({\mathcal F},\Gamma)}$. There are an even number of $E$-cuts in the boundary of each square in ${\mathcal K}_{({\mathcal F},\Gamma)}$.
\end{proposition}

\begin{proof}

\begin{lemma}
\label{facilemaisimportant}
Let $E$ be an horizontal EoE-cell of ${\mathcal K}_{({\mathcal F},\Gamma)}$ and let ${\mathcal C}$ be some $E$-cluster in ${\mathfrak P}(E)$ (see Definition \ref{oprotein}):
\begin{enumerate}

 \item Any two left ${\mathfrak H}_{\mathcal C}(E)$-translates of ${\mathcal C}$ are disjoint. For $h \in {\mathfrak H}_{\mathcal C}(E)$, $h.{\mathcal C}$ is separated from each closest cluster in ${\mathfrak H}_{\mathcal C}(E) . {\mathcal C}$ by exactly one horizontal essential cell belonging to both crowns. Two {\em adjacent} left ${\mathfrak H}_{\mathcal C}(E)$-translates of ${\mathcal C}$, i.e. which are separated by exactly one horizontal essential cell satisfy one of the following two assertions:
 
 \begin{itemize}
   \item either they  are the image one from the other by a conjugate of a $E$-horizontal generator in ${\mathfrak H}_{\mathcal C}(E)$, in which case the essential cell between the two is termed a {\em link},
   \item  or they are separated by the image of a bridge of the white $E$-block, in which case the essential cell between the two is termed a {\em bridge}.
\end{itemize}

 \item ${\mathcal T}_{\pi^\infty_k({\mathcal C})} \setminus {\mathfrak H}_{\mathcal C}(E).{\mathcal C}$ consists of:
 
 \begin{itemize} 
   \item The left ${\mathfrak H}_{\mathcal C}(E)$-translates of the trees based at the vertices in $\partial {\mathcal C}$ where the admissible color vanishes. 
    \item For any disabled pair of twins, the left ${\mathfrak H}_{\mathcal C}(E)$-translates of the ray contained in their $\F{n}$-axis starting at the vertex in $\partial \mathcal C$ with non-vanishing admissible color. 
 \end{itemize}
 \end{enumerate}
\end{lemma}

\begin{proof}
Any $E$-horizontal generator in ${\mathfrak H}_{\mathcal C}(E)$ is a shift so that it carries an incoming cell in $Cr(\mathcal C)$ to an outgoing one, or the converse. Hence, $h_{\gamma,\gamma^\prime}.{\mathcal C}$ is disjoint from $\mathcal C$ and connected to it by the essential cell $\gamma^\prime$. Moreover, if we denote by $h_{\gamma^-_j,\gamma^+_{j}}$, $j=1,2,\ldots$, the horizontal generators then all the $\gamma^-_j$ and $\gamma^+_j$ are distinct since their pairs belong to the partition of $Cr(\mathcal C)$ given by Lemma \ref{partition}. Item $(1)$ follows. Item $(2)$ is a straightforward consequence of the definitions.
\end{proof}

\begin{lemma}
\label{multiplicite}
With the assumptions and notations of Proposition \ref{0 ou 2}: let $m(\mathcal F,\Gamma)$ denote the maximal number of consecutive occurrences of a $\F{n}$-generator with same label in the reduced images of the $\F{n}$-generators under the $F_i$'s. Then the essential nucleus of any c-top in ${\mathfrak P}(E)$ is surrounded on each side in the c-top by at least $m(\mathcal F,\Gamma)-1$ essential cells with same label.
\end{lemma}

\begin{proof}
This is an immediate consequence of the definition of $Top(E)$.
\end{proof}

\begin{lemma}
\label{one}
With the assumptions and notations of Proposition \ref{0 ou 2}:

\begin{enumerate}

 \item For any link or bridge $X$ (see Lemma \ref{facilemaisimportant}) and any vertical element $t \in \F{k}$, neither $t^{-1} X t$ nor $t X t^{-1}$ contains the image of a nucleus in an orbit of a block-stair or lift.

 \item For any shift $h$ of a c-top in ${\mathfrak P}(E)$ and any shift $h^\prime$ of a c-bottom of the same block $\mathfrak M$, $h. {\mathfrak M}$ and $h^\prime . {\mathfrak M}$ are disjoint.
 
 \item The same assertion is true for a shift of a c-bottom or c-top and a mirror of ${\mathfrak P}(E)$ or two distinct mirrors of ${\mathfrak P}(E)$.

  \item The admissible coloring of ${\mathfrak P}(E)$ extends to a left ${\mathfrak H}(E)$-invariant admissible coloring of $\overline{{\mathfrak H}(E).{\mathfrak P}(E)}$ (the completion of ${\mathfrak H}(E).{\mathfrak P}(E)$). 

\end{enumerate}

\end{lemma}

\begin{proof}
Let us prove the first item. For the bridges, this is just a consequence of Corollary \ref{toutcapourca}. Let us consider the links coming from the shift at twins of the white block or its companion in ${\mathfrak P}(E)$. For the nuclei in the orbit of a lift, this comes from the fact that these shifts are always separated from such an orbit by a bridge and since $(\mathcal F,\Gamma)$ is one-sided Corollary \ref{toutcapourca} applies. For the nuclei in the orbit of a stair, observe that they are essential. Hence, from Lemma \ref{multiplicite}, if one considers the orbit of a vertex $g$ spanned by the shift, $t^{-1} g t$ and $t g t^{-1}$ will both be separated, in the corresponding horizontal tree, from the closest nucleus in the stair by at least $m(\mathcal F,\Gamma)$ essential cells with same label. By definition of $m(\mathcal F,\Gamma)$ this forbids $t^{-1} X t$ and $t X t^{-1}$ to contain a nucleus. Item $(1)$ is proved.

Item $(2)$ comes from Item $(1)$ and the fact that both shifts lie in distinct sides of the nucleus or center. Item $(3)$ comes from the fact that they are separated by at least one bridge and, since $(\mathcal F,\Gamma)$ is one-sided, Corollary \ref{toutcapourca} applies. Item $(4)$ comes from the fact that the considered shifts are enabled and there is exactly one mirror with non-vanishing color for each bridge (see Lemma \ref{partition}).
\end{proof}

\begin{lemma}
\label{two}
With the assumptions and notations of Proposition \ref{0 ou 2}:

\begin{enumerate}

  \item All the left translates of ${\mathfrak P}(E)$ under the elements of the subgroups generated by all the $E$-stairs are disjoint and are disjoint from ${\mathfrak H}(E).{\mathfrak P}(E)$.
  
  \item There is a left ${\mathfrak D}(E)$-invariant admissible color on $\overline{{\mathfrak D}(E).{\mathfrak P}(E)}$. 

\end{enumerate}
\end{lemma}

\begin{proof}
This is a consequence of the definitions and of Lemma \ref{multiplicite}. Just observe that ${\mathfrak D}(E)$ is the subgroup generated by ${\mathfrak H}(E)$ and $St(E)$.
\end{proof}

\begin{lemma}
\label{three}
With the assumptions and notations of Proposition \ref{0 ou 2}, let us consider the white $E$-block $\mathfrak M$ in ${\mathfrak P}(E)$. Let $h$ be a horizontal geodesic from the nucleus in a block-lift of $\mathfrak M$ to a closest nucleus of ${\mathfrak R}(E).{\mathfrak P}(E)$, coming from a nucleus in the companion in ${\mathfrak P}(E)$. If one colors the endpoints of $h$ as the endpoints of the initial nuclei in ${\mathfrak P}(E)$, then these admissible colors agree. 
\end{lemma}

\begin{proof}
This is a consequence of the fact that $(\mathcal F,\Gamma)$ is tied, one-sided and Corollary \ref{toutcapourca}.
\end{proof}

Let us now state a lemma which describes the ``fundamental'' $E$-biblock in terms of cuts:

\begin{lemma}
\label{fm}
With the assumptions and notations of Proposition \ref{0 ou 2} and Definition \ref{cut}: at the exception of the $i$-squares with bottom the nuclei of the c-tops of the white block ${\mathfrak M}$, any $i$-square in the cylinder over the $E$-biblock has one of the following types:
\begin{enumerate}
  \item It has two vertical elementary cuts in its boundary and none in its horizontal boundary.
  
  \item It has exactly one horizontal elementary cut and exactly one elementary vertical cut in its boundary.
  
  \item It has no elementary cut in its boundary.
\end{enumerate}

\end{lemma}

The following lemma is straightforward from Lemma \ref{fm}, \ref{lemme a la con} and the definition of the various lifts and stairs:

\begin{lemma}
\label{cfm}
With the assumptions and notations of Proposition \ref{0 ou 2} and Definition \ref{cut}: 
\begin{enumerate}
  \item The $i$-squares with bottom the nuclei $E^l_{top}$, $l \in \{1,\cdots,k\}$ of the c-tops of the white block have in their top the image of $E^l_{top}$ under the block-lift ${\mathcal L}_{i,E^l_{top}}$, which is a cut, and no vertical elementary cut in their boundary. The $l$-squares with bottom the center $E$ of the c-bottom in the white block have in their top the nuclei $E^l_{top}$, their bottom is the image of $E^l_{top}$ under the inverse of the block-lift ${\mathcal L}_{i,E}$, hence a cut and they have no vertical cut in their boundary.
  
  \item If one consider the images of ${\mathfrak P}(E)$ under the block-stairs as in Lemma \ref{lemme a la con} any $i$-square with $i \in I_E$ whose bottom in the crown of ${\mathfrak P}(E)$  belongs to an enabled pair of twins with non-vanishing color admits exactly two vertical cuts in its boundary and none in its horizontal one. Any $2$-cell whose bottom in the crown of ${\mathfrak P}(E)$  belongs to a pair of twins spanning a center or nucleus either satisfies the same property or admits no vertical cut in its boundary.
  
  \item If one considers a $i$-square with bottom a bridge of the $E$-biblock and $i \in I_E$, it inherits from the corresponding mirror or its inverse two vertical cuts in its boundary.
  
\end{enumerate}
   
\end{lemma}

We now complete the proof of Proposition \ref{0 ou 2}. Lemma \ref{two} gives a left ${\mathfrak D}(E)$-invariant admissible color on  $\overline{{\mathfrak D}(E).{\mathfrak P}(E)}$. Moreover, Lemma \ref{three} allows us to extend this admissible color, in an equivariant way, to the lifts of $\overline{{\mathfrak D}(E).{\mathfrak P}(E)}$. This does not mean at this point that one has a left-invariant admissible color on $\overline{{\mathfrak R}(E).{\mathfrak P}(E)}$ since adding the lifts to the generators of the previous subgroups adds elements in the horizontal $E$-stabilizer ${\mathfrak S}_{hor}(E)$ (see Lemma \ref{horEstab}).

Let us consider the nuclei of the companion block of ${\mathfrak P}(E)$. Since $(\mathcal F,\Gamma)$ is one-sided, the image in the corresponding c-top, under a block-lift of this block, of a horizontal geodesic $h$ as in Lemma \ref{three} leaves the c-top by the side where the admissible color vanishes. If one considers the nuclei of the white block, the image in the c-bottom under a block-lift of this block of a horizontal geodesic $h$ as in Lemma \ref{three} leaves the c-bottom by the side where there is no element of $Bot(E^l_{top})$ ($l \in \{1,\cdots,k\}$) hence also where the admissible color vanishes. 

We now have to check that, considering any two of these horizontal geodesics $h_0, h_1$ from (the image of) a nucleus in a block-lift-orbit to a block-stair-orbit, none of the two contains the (image of the) nucleus at the end of the other. The horizontal geodesic between a nucleus of the white block $\mathfrak M$ and its companion or an adjacent mirror image of this companion ${\mathfrak M}^\prime$ has the following form: let us consider the bridge $X$ between the two and without loss of generality assume that it is oriented from $\mathfrak M$ to ${\mathfrak M}^\prime$, and that it lies in the right-side of the center of the c-bottom of $\mathfrak M$. The geodesic from the c-bottom of the white block $\mathfrak M$ to the c-bottom of ${\mathfrak M}^\prime$ reads $R X R^\prime$ where $R^\prime = R^{-1}$ as an edge-path.Then, the reduction of $F_i(R X R^\prime)$ is the geodesic $h$ between the nucleus $E^i_{top}$ of $\mathfrak M$ and the corresponding nucleus of ${\mathfrak M}^\prime$. It decomposes in the following way:

\begin{itemize}
  \item $h$ starts with the subgeodesic $h_1$ of $F_i(R)$ after $E^i_{top}$: this subpath contains some $Y^i_{top}$ with $Y \in Bot^*_i(E)$ hence strictly higher than $E$.
  \item $h_1$ is followed by $F_i(X)$ which contains exactly once the highest $1$-cell $X^i_{top}$ of $h$ since, by Lemma \ref{analyse}, $X$ is the highest $1$-cell of $R X R^\prime$.
  \item $F_i(X)$ is followed by $h_2$ and, again by Lemma \ref{analyse}, either $h_1$ or $h_2$ contains the highest cell in $Bot^*_i(E)$ of the edge-path read by $R$ (this might be a $1$-cell with same label as $X$).
\end{itemize}
  
Finally, considering the reduced images or pre-images $h_t$ of $R X R^\prime$ under $F_t$ for $t \in \F{k}$ of length greater than $1$, one has to add to the image (or pre-image) of the following description the reduction of $m(\mathcal F,\Gamma)$ times the image of $E$, which in particular contains $m(\mathcal F,\Gamma)$ times a cell with the same label as $E$, for each increment of the length of $t$ by $1$.  

It follows from this description that, considering the images of such two distinct horizontal geodesics $h_t, h_{t^\prime}$ under a block-lift in a same c-top, or in the c-bottom of the white block, either  $h_t$ and $h_{t^\prime}$ diverge at some point or one is equal to the other, in which case the two nuclei at the end are identified, but none of the two might contain the nucleus at the end of the other. Thus there is no obstruction to extend equivariantly the admissible color and we so got the following

\begin{lemma}
\label{four}
With the assumptions and notations of Proposition \ref{0 ou 2}:

\begin{enumerate}
  \item There is a left ${\mathfrak R}(E)$-invariant admissible color on $\overline{{\mathfrak R}(E).{\mathfrak P}(E)}$.
  \item There are no other cut in ${\mathfrak P}(E)$ than the elementary cuts.
\end{enumerate}
\end{lemma} 

Lemmas \ref{fm} and \ref{cfm} give us that each $2$-cell of the cylinder over the union of the $E$-biblock with its crow inherit zero or two cuts in their boundary when applying once each lift and applying the stair generators as in Lemma \ref{lemme a la con}. Item $(1)$ of Lemma \ref{four} gives the preservation of the number of vertical cuts in each $2$-cell. Item $(2)$ of the same lemma gives the preservation of the number of horizontal cuts.
\end{proof}.

The following corollary is a straightforward consequence of Proposition \ref{0 ou 2}:

\begin{corollary}
\label{youyou}
Let $E$ be an horizontal EoE-cell of ${\mathcal K}_{({\mathcal F},\Gamma)}$. If $VH_R(E)$ (resp. $VH_L(E)$) denotes the set of all the vertices of ${\mathcal K}_{({\mathcal F},\Gamma)}$ which are connected to the right vertex of $E$ by an edge-path passing through an even (resp. odd) number of cuts then $(VH_L(E), VH_R(E))$ is a partition of the vertices of ${\mathcal K}_{({\mathcal F},\Gamma)}$, and of the elements of $G^\sigma_{n,k}$ in two disjoint, non-empty, components.
\end{corollary}

\subsection{Definition and stabilizer}

\begin{lemma}
The partition given by Corollary \ref{youyou} is left $G^\sigma_{n,k}$-invariant.
\end{lemma}

\begin{proof}
The construction of the white $E$-block only depends on the sets $Bot^i(E)$, and we obviously have $Bot^i(g.E) = g.Bot^i(E)$. The only choice in the construction of the $E$-biblock is the choice of a highest bridge which amounts to a choice of a mirror. Since there are all the other mirrors with positive admissible color in the replicative subgroup, this choice does not modify ${\mathfrak R}(E).{\mathfrak P}(E)$. 
\end{proof}

\begin{definition} 
\label{definition de mur diagonal}
Let $E$ be an horizontal EoE-cell of ${\mathcal K}_{({\mathcal F},\Gamma)}$. The {\em diagonal $E$-wall $W(E)$} is the partition $(VH_L(E),VH_R(E))$ given by Corollary \ref{youyou}. The {\em collection of all the diagonal walls} denotes the union, over all the EoE-cells in  ${\mathcal K}_{({\mathcal F},\Gamma)}$, of all the diagonal $E$-walls. The {\em label of $W(E)$} is the label of $E$.
\end{definition}

\begin{lemma} 
\label{stabilisateur diagonal}
With the assumptions and notations of Definition \ref{definition de mur diagonal}:  the left-stabilizer of $W(E)$ is the $E$-replicative subgroup ${\mathfrak R}(E)$.
\end{lemma}

\begin{proof}
The two sides of $W(E)$ are separated by the set of $E$-cuts which is left ${\mathfrak R}(E)$-invariant. \end{proof}

\subsection{Finiteness}

\begin{proposition} 
\label{finitude des diagonal}
Any two elements of $G^\sigma_{n,k}$ are separated by a finite number of diagonal walls.
\end{proposition}

\begin{proof}
Since any $g \in G^\sigma_{n,k}$ is connected to $e$ by a finite edge-path, it is sufficient to prove that any EoE-cell spans only a finite number of diagonal walls. Since the collection of walls is $G^\sigma_{n,k}$-invariant, any EoE-cell is the left $G^\sigma_{n,k}$-translate of an EoE-cell in a given fundamental region for the action of $G^\sigma_{n,k}$: there are a finite number of such $1$-cells. Moreover, a $1$-cell $E^\prime$ spans a diagonal $E$-wall if and only if $E^\prime$ is the left $g$-translate of some elementary $E$-cut with $g \in {\mathfrak R}(E)$ and there are a finite number of such elementary $E$-cuts. Assume that $E^\prime$ spans an infinite number of diagonal walls. We so get an infinite sequence of elements $g_{i_j} \in G^\sigma_{n,k}$ and of cells $E_{i_j}$ such that $E^\prime = g_{i_j} E_{i_j}$ where the $E_{i_j}$ are in the set $S$ formed by the union over the EoE-cells $E$ in the fundamental region of all the elementary $E$-cuts. But this set $S$ is finite. Thus, since the sequence $(g_{i_j})$ is assumed to be infinite, we get two distinct elements $g_{i_l}, g_{i_m}$ and an elementary cut $E_0$ with $g_{i_l} . E_0 = g_{i_m} . E_0 = E^\prime$. This implies $g^{-1}_{i_m} . g_{i_l}. E_0  = E_0$. This is impossible since the $1$-cells in  ${\mathcal K}_{({\mathcal F},\Gamma)}$ have trivial stabilizer for the left $G^\sigma_{n,k}$-action.
\end{proof}

\section{Conclusion for suspensions of unipotent free subgroups of $\Out{\F{n}}$}
\label{conclusion unipotent}

\begin{definition}
A {\em collection of vertizontal walls for ${\mathcal K}_{({\mathcal F},\Gamma)}$} is the union of the collection of vertical walls given by Definition \ref{mur vertical} with the collection of diagonal walls given by Definition \ref{definition de mur diagonal}. 
\end{definition}

Lemmas \ref{stabilisateur vertical}, \ref{stabilisateur diagonal} on one hand, Propositions \ref{vertical} and \ref{finitude des diagonal} on the other one give the following

\begin{proposition}
\label{action}
If $\mathcal W$ is a collection of vertizontal walls for ${\mathcal K}_{({\mathcal F},\Gamma)}$, $({\mathcal K}_{({\mathcal F},\Gamma)},\mathcal W)$ is a space with walls structure for the left-action of $G^\sigma_{n,k}$ on ${\mathcal K}_{({\mathcal F},\Gamma)}$. 
\end{proposition}

We now prove the following result: 

\begin{proposition}
\label{proprete}
If $\mathcal W$ is a collection of vertizontal walls for ${\mathcal K}_{({\mathcal F},\Gamma)}$, the left $G^\sigma_{n,k}$-action on the space with walls structure $({\mathcal K}_{({\mathcal F},\Gamma)},\mathcal W)$ is proper.
\end{proposition}

\begin{proof}

Any $g \in G^\sigma_{n,k}$ is uniquely written as $h.v$ where $h \in \F{n}$ is a horizontal element and $v \in \F{k}$ a vertical one. 

\begin{lemma}
\label{trivialite}

Let $h_i v_i \in G^\sigma_{n,k}$ be an infinite sequence of elements such that the lengths of the $v_i$ tends toward infinity with $i$. Then the number of vertizontal walls which separate $e$ from $h_i v_i$ also tends toward infinity with $i$.
\end{lemma}

\begin{proof}
The vertical walls, which form a subset of the vertizontal walls, are the classical walls for the vertical free group $\F{k}$, hence the conclusion.
\end{proof}

It is thus sufficient to prove that for any infinite sequence of horizontal elements $h_i$ in $G^\sigma_{n,k}$ whose lengths tend toward infinity with $i$, the number of vertizontal walls separating $e$ from $h_i$ also tends toward infinity with $i$. Any horizontal element $h_i$ is represented in the $1$-skeleton of ${\mathcal K}_{({\mathcal F},\Gamma)}$ by a horizontal geodesic starting at $e$ reading $h_i$. 

\begin{lemma}
\label{un oubli}
Let $(h_l)$ be any infinite sequence of horizontal elements and let $(\omega_l)$ be the associated sequence of horizontal geodesics starting at $e$. The number of vertizontal walls spanned by $(h_l)$ tends toward infinity with the number of vertizontal walls spanned by $(\omega_l)$. 
\end{lemma}

\begin{proof}
If $C$ is the diameter of a fundamental region, any geodesic of length $C$ crosses at least once an EoE-cell. Moreover, by Proposition \ref{0 ou 2}, the parity of the number of $E$-cuts crossed by a horizontal geodesic only depends on the element it defines since any relator contains two cuts in its boundary. Hence the conclusion.
\end{proof}

We are now going to prove that the number of diagonal walls spanned by a horizontal geodesic is bounded below by a constant times the length of the geodesic. The easy case is the case of the {\em topmost diagonal walls}:

\begin{lemma}
\label{il fallait y penser}
 Let $E$ be an EoE-cell. If $E$ is topmost then there are no other cut of the diagonal wall $W(E)$ in the same horizontal tree. In particular, any horizontal geodesic which contains a topmost EoE-cell $E$ spans the diagonal wall $W(E)$ and so a vertizontal wall.
\end{lemma}

\begin{proof}
If $E$ is topmost, the only horizontal elementary cut for $W(E)$ is $E$ and the stabilizer of $W(E)$ is the replicative subgroup ${\mathfrak R}(E)$ which is generated by the $(i,E)$-lifts, exactly one for each $i \in \{1,\cdots,k\}$. Since each one has the form $h_i t_i$, the lemma follows. 
\end{proof}

\begin{lemma}
\label{easy}
Let $E$ be an EoE-cell which is not $i$-topmost for some $i \in \{1,\cdots,k\}$. Then $E$ is separated from any other horizontal $E$-cut of $W(E)$ by an EoE-cell strictly higher than $E$.
 
 \end{lemma}

 \begin{proof}
This is a consequence of the fact that all the $E$-horizontal generators are separated from $E$ by a bridge which has greater height than $E$. For the other generators of ${\mathfrak S}_{hor}(E)$ this is clear.
 \end{proof}

\begin{lemma}
\label{C}
There exists a constant $C \geq 1$ such that, if $\omega$ is a horizontal geodesic of length greater than $C$ with first EoE-cell $E$, then if the geodesic subsegment of length $C$ starting with $E$ in $\omega$ does not contain a EoE-cell strictly higher than $E$ then $\omega$ spans the diagonal wall $W(E)$.
\end{lemma}

\begin{proof}
Let $C$ be the maximum of the diameters of all the $E$-blocks, for the EoE-cells $E$ in a fundamental region for the action of $G^\sigma_{n,k}$. Assume that the initial segment of length $C$ in $\omega$ does not contain an EoE-cell strictly higher than $E$. Then Lemma \ref{C} is a direct consequence of Lemma \ref{easy}.
\end{proof}

We now complete the proof of Proposition \ref{proprete}. Let $C$ be the constant given by Lemma \ref{C}. Let $\omega$ be a horizontal geodesic with first EoE-cell $E$. By Lemma \ref{C}, either at distance at most $EoE(\Gamma) . C$ in $\omega$ we find some topmost cut $E^\prime$ or at distance at most $(EoE(\Gamma) -1). C$ we find a EoE-cell $E^\prime$ for which there is no higher EoE-cell further in $\omega$ at distance less than $C$. In the first case, Lemma \ref{il fallait y penser} gives that $\omega$ spans $W(E^\prime)$, in the second case, Lemma \ref{C} gives the same conclusion. We then iterate by considering the first EoE-cell after $E^\prime$ which is not a cut for the walls crossed between $E$ and $E^\prime$. We find such a cut at distance at most $EoE(\Gamma).C$ from $E^\prime$ in $\omega$. We so eventually get that the total number of diagonal walls spanned by $\omega$, and so the vertizontal wall distance between its endpoints, is at least $1/(2 EoE(\Gamma).C)$ times the length of $\omega$, that is the distance between its endpoints. Proposition \ref{proprete} follows from Lemma \ref{un oubli} and from Lemma \ref{trivialite} which gives the analogous statement for the vertizontal wall distance with respect to the vertical $\F{k}$-distance.
\end{proof}

Propositions \ref{action} and \ref{proprete} together give the following

\begin{theorem}
\label{interlude}
The suspension of a finite rank free group by a tied, one-sided unipotent finite rank free subgroup of outer automorphisms is a-T-menable in the sense of Gromov.
\end{theorem}

\begin{proposition} 
\label{allez}
Let $({\mathcal K}_{({\mathcal F},\Gamma)},\mathcal W)$ be the space with walls structure for $G^\sigma_{n,k}$ as given by Proposition \ref{action}. Then $dim(Cb({\mathcal W})) < \infty$ (see Theorem \ref{ceci est un rappel}). 
\end{proposition}

\begin{proof}

\begin{lemma}
\label{2}
A collection of vertizontal walls which pairwise cross contains at most one vertical wall and at most one topmost diagonal wall.
\end{lemma}

\begin{proof}
Vertical walls have only one cut in each vertical tree, and topmost diagonal walls have only one cut in each horizontal tree, hence the conclusion.
\end{proof}

We can thus deal only with collections of diagonal walls defined by non topmost EoE-cells.

\begin{lemma}
\label{contrainte horizontale}
Assume that two non-topmost diagonal walls $W_1,W_2$ with same label cross but that their intersections $W_i \cap {\mathcal T}_w$ with some horizontal tree ${\mathcal T}_w$ do not cross in ${\mathcal T}_w$. Then, any third distinct wall $W_3$ with same label such that $\{W_1,W_2,W_3\}$ is a collection of pairwise-crossing walls satisfies that $W_3$ and $W_i$, $i=1,2$, cross in any horizontal tree.
\end{lemma}

\begin{proof}
 Up to permutation of the indices, since $W_1 \cap {\mathcal T}_w$ and $W_2 \cap {\mathcal T}_w$ do not cross in ${\mathcal T}_w$, there exists a cut $E_1$ of $W_1$ such that all the cuts of $W_2$ in ${\mathcal T}_w$ lie in a same side of $E_1$. Since the diagonal walls $W_1$ and $W_2$ share a same label and since they cross in ${\mathcal K}_{({\mathcal F},\Gamma)}$, this implies that the stair-orbit of one cut intersects the lift-orbit of the other.  If $E$ denotes this cut, there is only one wall for which the lift-orbit of $E$ is a set of cuts (the wall $W(E)$) and one wall for which the stair-orbit of $E$ is a set of cuts. Therefore, if a third distinct wall $W_3$ crosses both $W_1$ and $W_2$ then $W_3$ and $W_i$ ($i=1,2$) cross in any horizontal tree.
\end{proof}

\begin{lemma}
\label{vi}
The maximal number of non-topmost diagonal walls with same $\Gamma$-label which pairwise cross in a same horizontal tree is finite.
\end{lemma}. 

\begin{proof}
Let us consider two non-topmost EoE-cells $E_0$ and $E_1$ with same $\Gamma$-label in ${\mathcal T}_w$. Assume that the walls $W_i = W(E_i)$, $i=0,1$, cross in ${\mathcal T}_w$. Then there exits cuts $c^1_i, c^2_i$ of $W_i$ in some horizontal geodesic  $u$ such that, up to permuting $W_0$ and $W_1$, $u=c^1_0 \cdots c^1_1 \cdots c^2_0 \cdots c^2_1$. Assume that neither $c^2_0$ is a left ${\mathfrak S}_{hor}(E_0)$-translate of $c^1_0$ nor $c^2_1$ is a left ${\mathfrak S}_{hor}(E_1)$-translate of $c^1_1$. This implies that the geodesics connecting them lie over some bridges: such geodesics contain exactly $2$-cells of the height of the bridge which is the height of the geodesic. This implies that the maximal number of walls for which one might have such overlapping of geodesics is finite. We now assume that there is $i \in \{0,1\}$ such that $c^2_i$ is a left ${\mathfrak S}_{hor}(E_i)$-translate of $c^1_i$. Then, since ${\mathfrak S}_{hor}(E_i)$ is finitely generated and so quasi-convex, a left-translation by an element of  ${\mathfrak S}_{hor}(E_i)$ carries $c^2_i$ to a bounded neighborhood of the $c^2_{i+1}$-biblock, where the indices are written modulo $2$. Since biblocks are finite, there are only a finite number of walls possible.
\end{proof}

Since the number of distinct $\Gamma$-labels of walls is equal to $EoE(\Gamma)$, and so finite, Lemma \ref{contrainte horizontale} and Lemma \ref{vi} give that there is a finite number of non-topmost diagonal walls in a family of vertizontal walls which pairwise cross. Together with Lemma \ref{2}, this gives Proposition \ref{allez}. 
\end{proof}

Propositions \ref{action}, \ref{proprete} and \ref{allez}, together with Theorem \ref{ceci est un rappel} give the following

\begin{theorem}
\label{blablabla}
The suspension of a finite rank free group by a tied, one-sided unipotent finite rank free subgroup of outer automorphisms $G^\sigma_{n,k}$ acts properly on some finite-dimensional cube complex $Cb(\mathcal W)$ where $({\mathcal K}_{({\mathcal F},\Gamma)},\mathcal W)$ is the space with walls structure given by Proposition \ref{action}.
\end{theorem}

Let us conclude with a lemma about the minimal dimension of the above cube complex:

\begin{lemma}
\label{minimal bound}
Let $({\mathcal K}_{({\mathcal F},\Gamma)},\mathcal W)$ be the space with walls structure for $G^\sigma_{n,k}$ as given by Proposition \ref{action}. 
If $M$ is the maximal number of EoE-cells sharing a same bridge, in the same side of this bridge, then $dim(Cb({\mathcal W})) \geq 2 M+2$. This lower-bound is sharp.
\end{lemma}

\begin{proof}
By Theorem \ref{ceci est un rappel}, $dim(Cb({\mathcal W}))$ is the supremum of the cardinalities of family of walls which pairwise cross. By construction, there are two walls crossed by a same non-topmost EoE-cell $E$: the wall $W(E)$ and the wall $W(E^\prime)$ where $E^\prime$ is the nucleus of the companion of the white $E$-block. These two walls cross: this is easily checked as it is easily checked that all the walls associated to cuts incident to a same bridge pairwise cross. Now add a vertical wall defined by a vertical cell incident to the horizontal tree containing all the preceding horizontal cells, and the wall defined by the bridge and you get a collection of walls which pairwise cross, hence the lemma. It is impossible to  increase this number since it is attained for instance when considering the automorphism $\alpha : \F{2} \rightarrow \F{2}$ defined by $\alpha(x_1)=x_1$ and $\alpha(x_2) = x_2 x_1$.
\end{proof}

\section{Conclusion for suspensions of free subgroups of outer automorphisms of $\F{n}$ admitting a tied, one-sided finite-index unipotent subgroup}

\begin{lemma} 
\label{bof1}
Given any two groups $G$ and $H$ together with a monomorphism $\sigma \colon G \hookrightarrow \Out{H}$,
if $G_0$ is a finite index subgroup of $G$ then $H \rtimes_\sigma G_0$ is a finite index subgroup of $H \rtimes G$.
\end{lemma}

\begin{proof}
Any element of $H \rtimes_\sigma G$ is uniquely written as $gh$, $g \in G$ and $h \in H$. Since $G_0$ has finite index in $G$,
there exists a finite number of elements $g_1,\cdots,g_r$ such that any $g \in G$ is in some left-class $g_i G_0$. Then for any element
$gh$ of $H \rtimes_\sigma G$, there is some $i$ such that $gh \in g_i G_0 h$. Any element in $G_0 h$ is in the semi-direct product $H \rtimes_\sigma G_0$
hence the conclusion.
\end{proof}

%
%

\begin{lemma} 
\label{bof2}
If $G$ is a group admitting a finite-index subgroup $G_0$ which acts properly isometrically on some finite dimensional
CAT(0) cube complex then so does $G$, on a cube complex of the same dimension.
\end{lemma}

\begin{proof}
A finite covering of a $j$-dimensionnal CAT(0) cube complex is a $j$-dimensionnal CAT(0) cube complex. By lifting the given action of $G_0$, one gets the lemma.
\end{proof}

\begin{proof}[Proof of Theorem \ref{the theorem}]

We consider the group $G^\sigma_{n,k} := \F{n} \rtimes_\sigma \F{k}$ with $\sigma(\F{k})$ a subgroup of $\Out{\F{n}}$ admitting a tied, one-sided finite-index unipotent subgroup $\mathcal U$. By Lemma \ref{bof1}, since $\mathcal U$ has finite-index in $\F{k}$, $\F{n} \rtimes \mathcal U$ is of finite-index in $G^\sigma_{n,k}$. By Theorem \ref{interlude}, since a-T-menability is preserved up to finite index, $G^\sigma_{n,k}$ is a-T-menable. By Theorem \ref{blablabla}, $\F{n} \rtimes \mathcal U$
acts properly isometrically on some finite-dimensional cube complex. Lemmas \ref{bof1} and \ref{bof2} then give the action of $G^\sigma_{n,k}$ on a cube complex of the same dimension (and Lemma \ref{minimal bound} gives a sharp minimal  bound for the dimension of this cube complex).
\end{proof}

\section{Examples}

\label{exemples}

\subsection{The Formanek-Procesi group with $n=3$}

The definition has been recalled in the introduction. A BFH-representative is the rose with $3$ petals $\mathcal R_3$, where each petal is identified to a different generator $x_i$. The filtration
is given by $\{x_1\} \subset \{x_1,x_2\} \subset \{x_1,x_2,x_3\}$ (we only list the edges of each graph of the filtration, every one contains
the unique vertex  of $\mathcal R_3$). The filtered homotopy equivalences are obtained by setting $f_j(x_i)$ to be the
edge-path reading $\sigma(t_j)(x_i)$.

There is exactly one topmost edge: the edge $x_3$. Figure \ref{FPhorizontal} shows the cuts for a piece of a diagonal wall with label $x_3$. The thick edges are the cuts.

\begin{figure}[htbp]
{\centerline{\includegraphics[height=6cm, viewport = 150 560 460 750,clip]{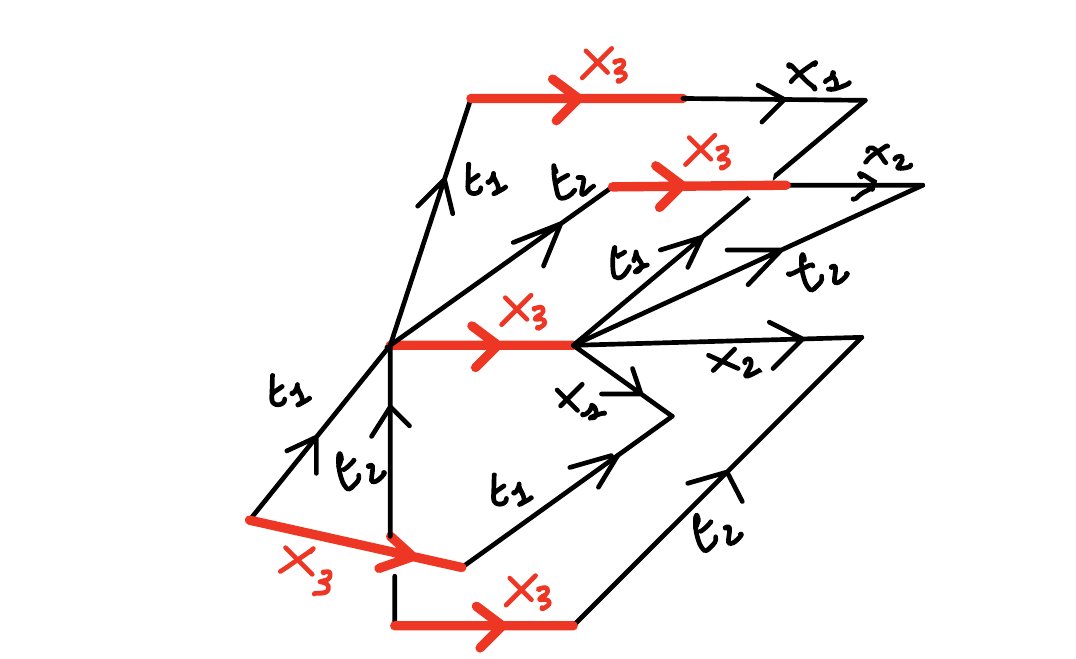}}} \caption{\label{FPhorizontal} Cuts defining a diagonal wall with label $x_3$ for the FP-group }
\end{figure}

There is in addition $x_2$ as $1$-topmost edge and $x_1$ as $2$-topmost edge. We refer the reader to Figure \ref{FP180621} for an illustration of the cuts for a wall with label $x_1$. 

\begin{figure}[htbp]
{\centerline{\includegraphics*[height=5cm,viewport= 70 455 550 735 clip]{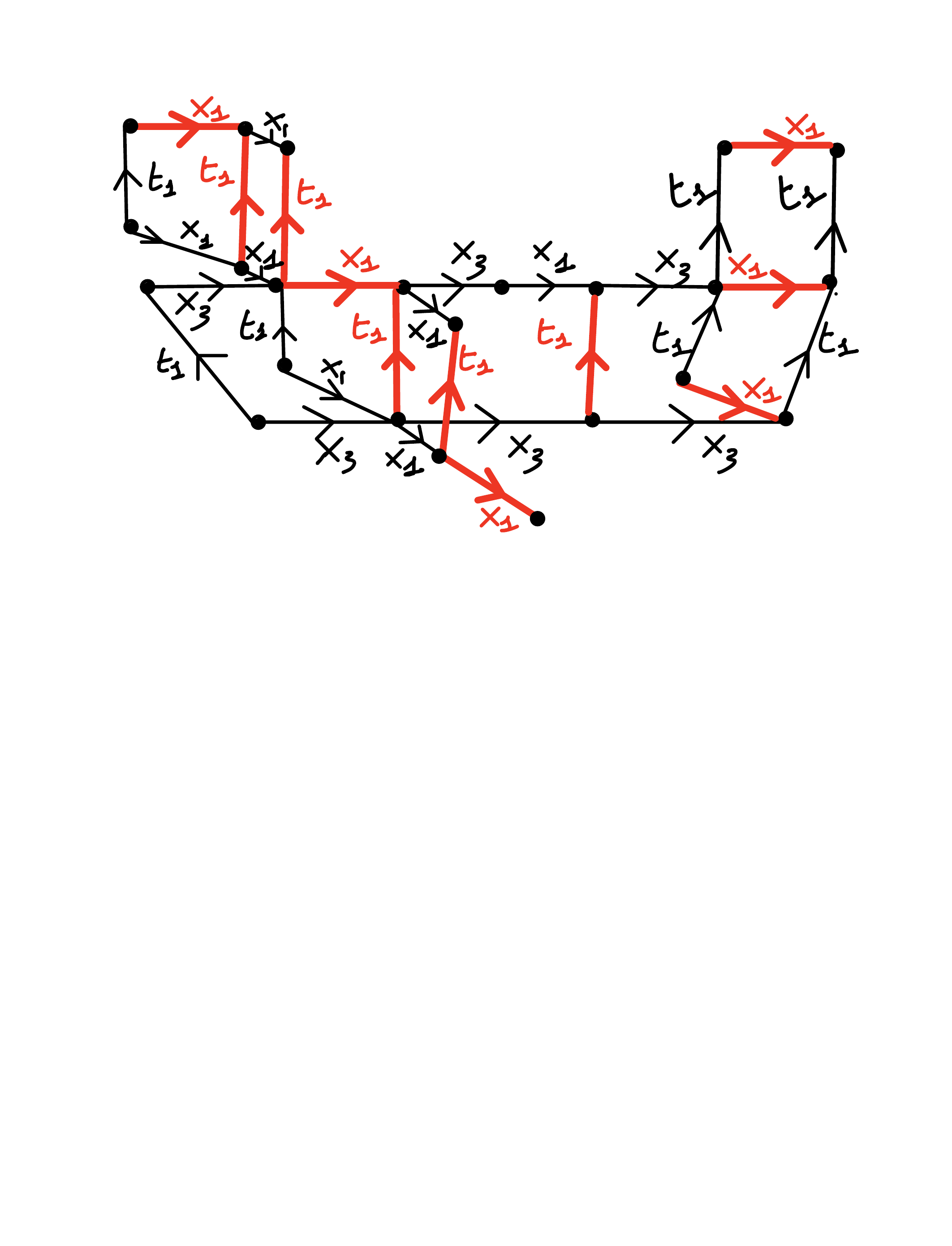}}} \caption{\label{FP180621} Cuts defining a diagonal wall with label $x_1$ for the FP-group}
\end{figure}

\begin{figure}[htbp]
{\centerline{\includegraphics*[height=7cm,viewport= 80 325 465 725 clip]{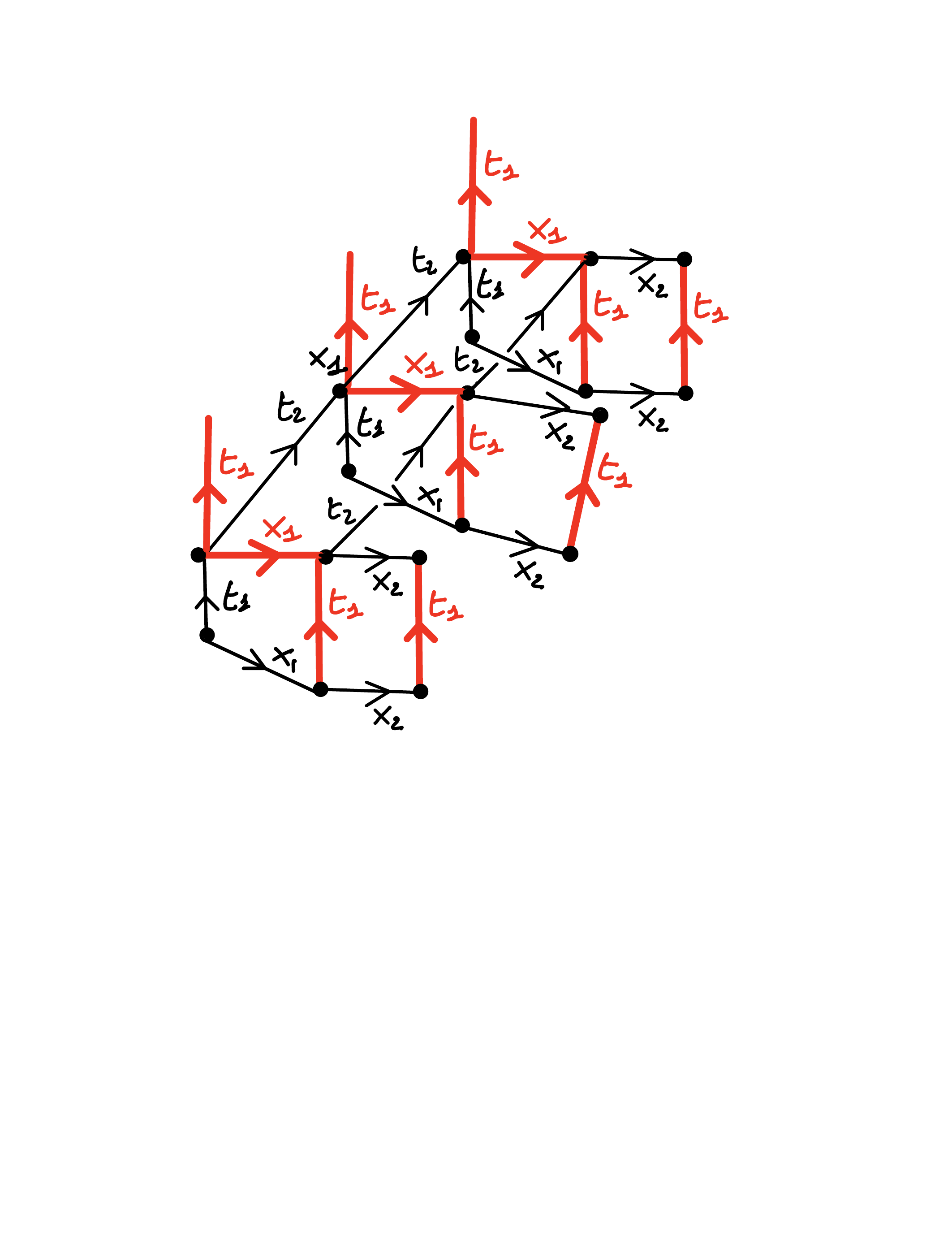}}} \caption{\label{tx2} Translates of the cuts of the companion block}
\end{figure}

We do not represent in the first picture the left-translates of the cuts by (conjugates of) $t_2$, nor the left-translates of the cuts in the companion of the white block (the cuts which are found at the left-side of this picture) by (conjugates of) $x_2$. The reader will find these other left-translates in Figures \ref{t2} (translates of the white block by conjugates of $t_2$) and \ref{tx2} (translates of the companion block by conjugates of $t_2$ and $x_2$).

\begin{figure}[htbp]
{\centerline{\includegraphics*[height=10cm,viewport= 210 310 420 735 clip]{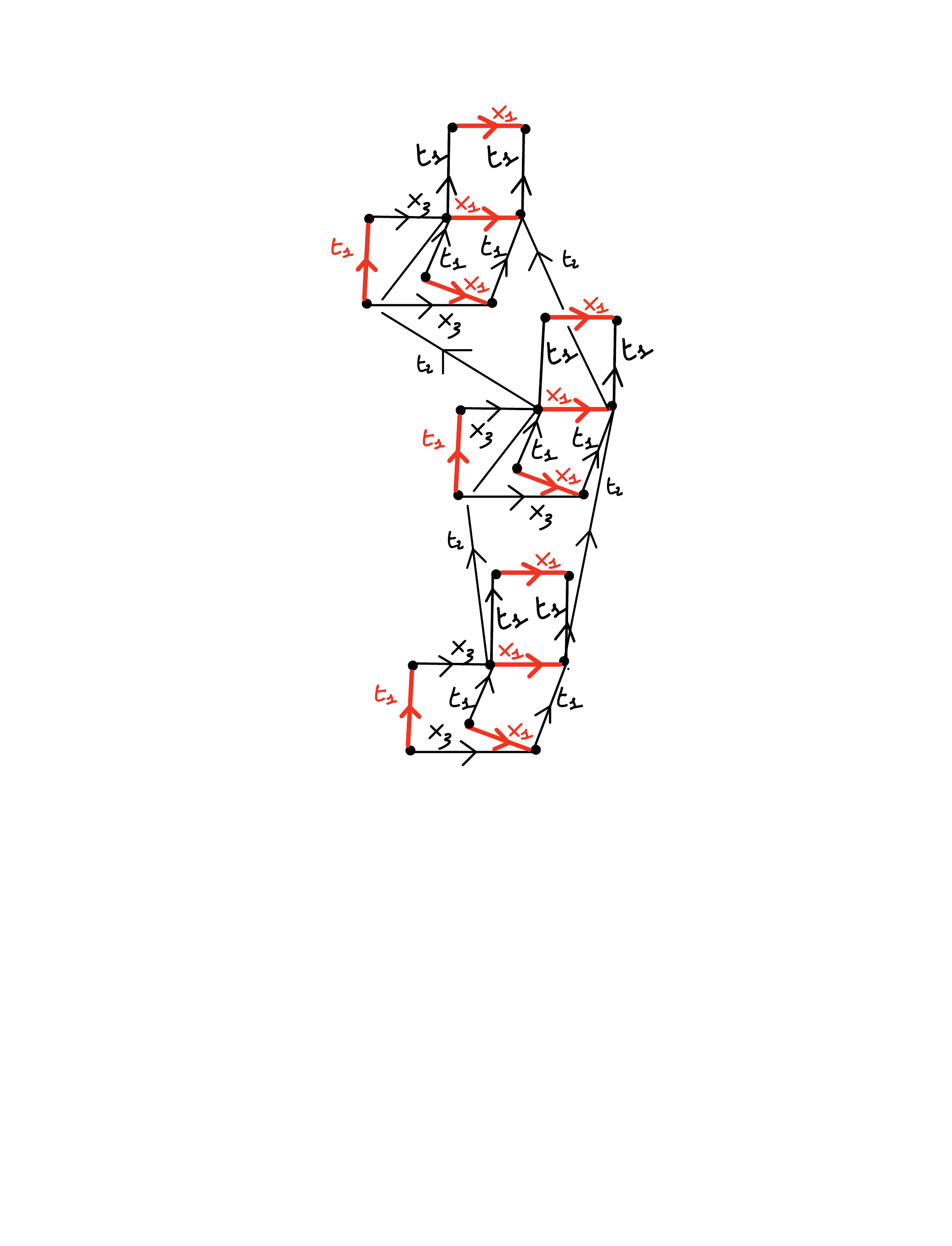}}} \caption{\label{t2} Translates of the cuts of the white block}
\end{figure}


\newpage

\subsection{An extension of $\F{4}$}

Here $G = \F{4} \rtimes_\sigma \F{2}$ with $\F{4} = \langle x_1,x_2,x_3,x_4 \rangle$, $\F{2} = \langle t_1,t_2 \rangle$ and $\alpha_i := \sigma(t_i)$ defined by:

$$\begin{array}{ccccc}
\alpha_1 & \colon & x_4 & \mapsto & x_4x_1 \\
& & x_3 & \mapsto & x_3 x_1 \\
& & x_2 & \mapsto & x_2 \\
& & x_1 & \mapsto & x_1 \\
\end{array} \begin{array}{ccccc}
\alpha_2 & \colon & x_4 & \mapsto & x_4x_3x_2x_1 \\
& & x_3 & \mapsto & x_3 x_2 x_1 \\
& & x_2 & \mapsto & x_2  \\
& & x_1 & \mapsto & x_1 \\
\end{array}$$ 

As in the previous example, the BFH-representative is the rose, whose petals are identified to the generators $x_i$. This is a one-sided, tied subgroup of polynomially growing automorphisms. The $1$-topmost edges are $x_2,x_3,x_4$ whereas the only $2$-topmost edge is $x_4$. Figure \ref{rang4} illustrates what looks like (a piece of) the white block and the companion.

\begin{figure}[htbp]
{\centerline{\includegraphics*[height=10cm,viewport=50 270 500 720 clip]{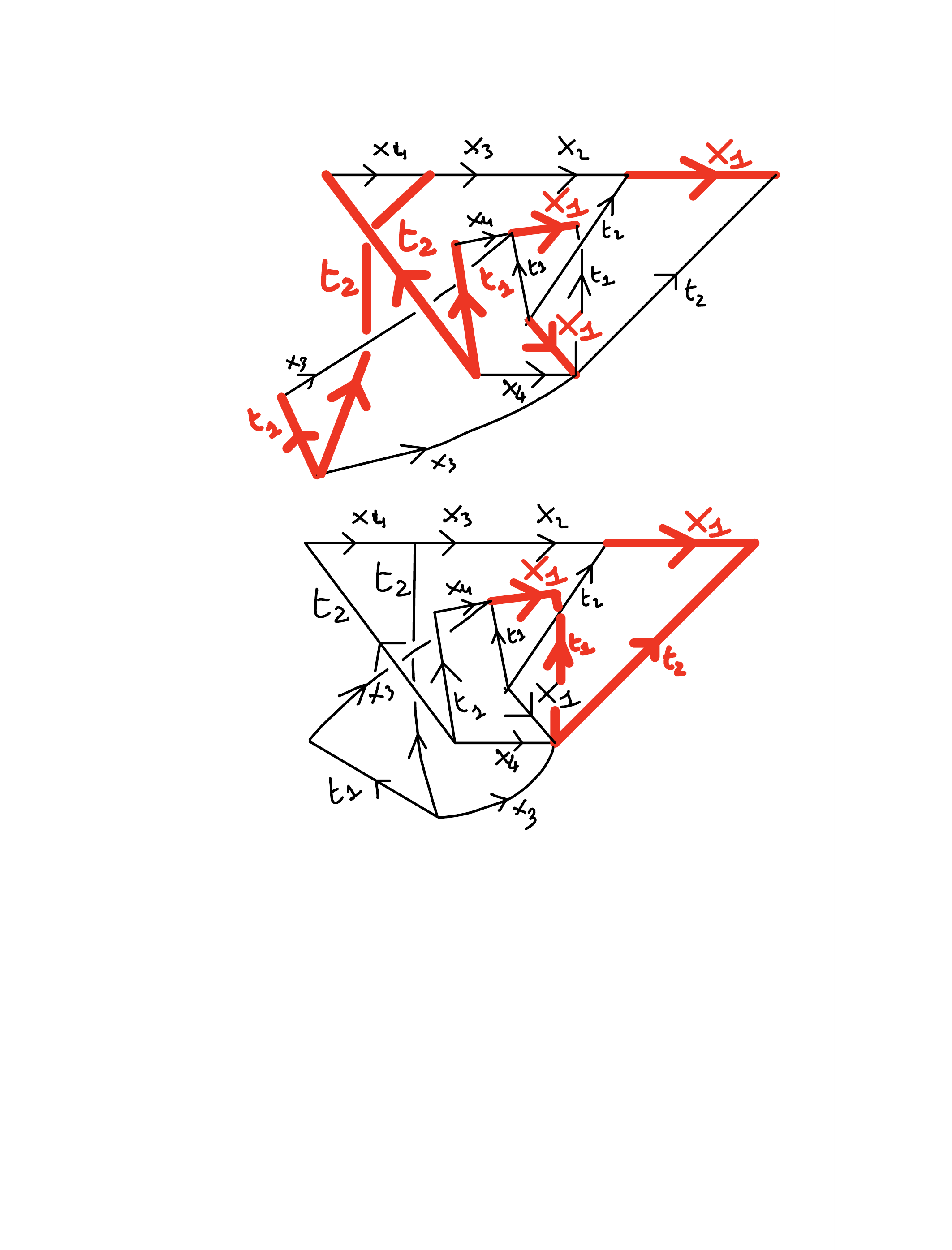}}} \caption{\label{rang4} White block and companion with label $x_1$ for the $\F{4}$-example}
\end{figure}

\bigskip

\noindent {\em Acknowledgements:} The author would like to warmly thank P.A.~ Cherix (Universit\'e de Gen\`eve, Geneva) who introduced the author to the
Haagerup property.  G.N.~Arzhantseva (Universit\"at Wien, Vienna) and P.A.~Cherix evoked the question of the a-T-menability of free-by-free,
and surface-by-free groups. At that time in Geneva, the interest of W.~Pitsch (U.A.B., Barcelona) in this question was a source of motivation. For these reasons,
all three of them deserve the gratitude of the author. He is moreover indebted to G.N.~Arzhantseva for telling him about the question of the a-T-menabilty of free-by-free
groups. Finally many thanks are also due to I. Chatterji (Universit\'e C\^ote d'Azur, Nice) for her patience when listening the numerous tries of the author to get a more general result.

\bibliographystyle{plain}
\bibliography{biblioHaagerup}

\begin{thebibliography}{10}

\bibitem{AW}
Charles~A. Akemann and Martin~E. Walter.
\newblock Unbounded negative definite functions.
\newblock {\em Canadian Journal of Mathematics. Journal Canadien de
  Math\'ematiques}, 33(4):862--871, 1981.

\bibitem{BFHexponentiel}
Mladen Bestvina, Mark Feighn, and Michael Handel.
\newblock The {T}its alternative for {${\rm Out}(F\sb n)$}. {I}. {D}ynamics of
  exponentially growing automorphisms.
\newblock {\em Annals of Mathematics}, 151(2):517--623, 2000.

\bibitem{BFHpolynomial}
Mladen Bestvina, Mark Feighn, and Michael Handel.
\newblock The {T}its alternative for {${\rm Out}(F\sb n)$}. {II}. {A} {K}olchin
  type theorem.
\newblock {\em Annals of Mathematics}, 161(1):1--59, 2005.

\bibitem{Bridson}
Martin~R. Bridson and Andr{\'e} Haefliger.
\newblock {\em Metric spaces of non-positive curvature}, volume 319 of {\em
  Grundlehren der Mathematischen Wissenschaften [Fundamental Principles of
  Mathematical Sciences]}.
\newblock Springer-Verlag, Berlin, 1999.

\bibitem{Burger}
M.~Burger.
\newblock Kazhdan constants for {${\rm SL}(3,{\bf Z})$}.
\newblock {\em Journal f\"ur die Reine und Angewandte Mathematik}, 413:36--67,
  1991.

\bibitem{ChatterjiNiblo}
Indira Chatterji and Graham Niblo.
\newblock From wall spaces to {$\rm CAT(0)$} cube complexes.
\newblock {\em International Journal of Algebra and Computation},
  15(5-6):875--885, 2005.

\bibitem{Cherix}
Pierre-Alain Cherix, Michael Cowling, Paul Jolissaint, Pierre Julg, and Alain
  Valette.
\newblock {\em Groups with the {H}aagerup property}, volume 197 of {\em
  Progress in Mathematics}.
\newblock Birkh\"auser Verlag, Basel, 2001.

\bibitem{delaHarpeValette}
Pierre de~la Harpe and Alain Valette.
\newblock La propri\'et\'e {$(T)$} de {K}azhdan pour les groupes localement
  compacts (avec un appendice de {M}arc {B}urger).
\newblock {\em Ast\'erisque}, (175):158, 1989.

\bibitem{Gautero}
F.~Gautero.
\newblock A non-trivial example of a free-by-free group with the haagerup
  property.
\newblock {\em Groups, Geometry and Dynamics}, 2012.

\bibitem{GH}
Erik Guentner and Nigel Higson.
\newblock Weak amenability of {$\rm CAT(0)$}-cubical groups.
\newblock {\em Geometriae Dedicata}, 148:137--156, 2010.

\bibitem{Haagerup}
Uffe Haagerup.
\newblock An example of a nonnuclear {$C\sp{\ast} $}-algebra, which has the
  metric approximation property.
\newblock {\em Inventiones Mathematicae}, 50(3):279--293, 1978/79.

\bibitem{HaglundPaulin}
Fr{\'e}d{\'e}ric Haglund and Fr{\'e}d{\'e}ric Paulin.
\newblock Simplicit\'e de groupes d'automorphismes d'espaces \`a courbure
  n\'egative.
\newblock In {\em The {E}pstein birthday schrift}, volume~1 of {\em Geom.
  Topol. Monogr.}, pages 181--248 (electronic). Geom. Topol. Publ., Coventry,
  1998.

\bibitem{Jolissaint}
Paul Jolissaint.
\newblock Borel cocycles, approximation properties and relative property {T}.
\newblock {\em Ergodic Theory and Dynamical Systems}, 20(2):483--499, 2000.

\bibitem{NibloRoller}
Graham~A. Niblo and Martin~A. Roller.
\newblock Groups acting on cubes and {K}azhdan's property ({T}).
\newblock {\em Proceedings of the American Mathematical Society},
  126(3):693--699, 1998.

\bibitem{Nica}
Bogdan Nica.
\newblock Cubulating spaces with walls.
\newblock {\em Algebraic \& Geometric Topology}, 4:297--309 (electronic), 2004.

\bibitem{Sageev}
Michah Sageev.
\newblock Ends of group pairs and non-positively curved cube complexes.
\newblock {\em Proceedings of the London Mathematical Society}, 71(3):585--617,
  1995.

\end{thebibliography}

\end{document}